\documentclass[11pt,a4paper]{article}	
\usepackage[utf8]{inputenc} 
\usepackage[english]{babel}
\usepackage{amsfonts}
\usepackage{amsmath}
\usepackage{amssymb}
\usepackage{amsthm}
\usepackage{float}
\usepackage{cases}
\usepackage{color}
\usepackage{dsfont} 
\usepackage{enumerate}
\usepackage{enumitem}
\usepackage{etex}
\usepackage {fancyhdr}
\usepackage{lastpage}
\usepackage{latexsym}
\usepackage{hyperref}
\usepackage{pgf,tikz}
\usetikzlibrary{arrows}
\usepackage{pifont}
\usepackage{pstricks}
\usepackage{stmaryrd}
\usepackage[top=3.9cm, bottom=3.5cm, left=3cm , right=3cm]{geometry}

\theoremstyle{plain}
\newtheorem{theorem}{Theorem}[section]
\newtheorem{lemma}[theorem]{Lemma}
\newtheorem{proposition}[theorem]{Proposition}
\newtheorem{corollary}[theorem]{Corollary}
\newtheorem{definition}[theorem]{Definition}

\theoremstyle{definition}
\newtheorem{hyp}{Assumption}
\theoremstyle{remark}
\newtheorem{remark}{Remark}[section]
\numberwithin{equation}{section}
\newcommand{\bol}[1]{\textbf{#1}}

\def\cC{{\mathcal C}}
 \def\d{{\rm d }}
 \def\F{{\mathcal F}}
\def\G{{\mathcal G}}
\def\J{{\mathcal J}}
\def\L{{\mathsf L}}
\def\N{{\mathbb N}}
\def\E{{\mathsf E}}
\def\P{{\mathsf P}}
\def\cP{{\mathcal P}}
\def \bolQ{Q}
\def\R{{\mathsf R}}
\def\cR{{\mathcal R}}
\def\U{{\cal U}}
\def\cX{{\mathcal X}}

\def \bolG{G}
\def \bolN{N}
\def\bbolN{{\bar N}}

\def \Ndemeps{N^{\rm \dem, \epsilon}}
\def\bolmu{\mu}
\def\bolei{e_i}
\def\bolej{e_j}
\def\sw{{\rm  sw}}
\def\dem{{\rm  dem}}
\def\bolnu{\nu} 
\def\bollambda{\lambda}

\def\cal{\mathcal} 
\def\ind{{\bf 1}}
\def \indi{\mathds{1}}

\newcommand{\rmi}{{\rm (i) $\hspace{1mm}$}}

\newcommand{\rmii}{{\rm (ii) $\hspace{1mm}$}}
\newcommand{\rmiii}{{\rm (iii)$\hspace{1mm}$}}

\newcommand{\rma}{{\rm a) $\>\>$}}
\newcommand{\rmb}{{\rm b) $\>\>$}} %$\hspace{1.5mm}$}}
\newcommand{\rmc}{{\rm c)$\>\>$}}

\usepackage{natbib}

\title{ Birth Death Swap population in random environment and aggregation with two timescales. }

\newcommand{\footremember}[2]{%
    \footnote{#2}
    \newcounter{#1}
    \setcounter{#1}{\value{footnote}}%
}

\author{%
Sarah Kaakai \footremember{LMM}{Laboratoire Manceau de Mathématiques (LMM) -  EA 3263, Le Mans Université, Avenue Olivier Messiaen
72085 Le Mans, France;  Email: {sarah.kaakai@univ-lemans.fr}.}%
  \and Nicole El Karoui\footremember{LPSM}{Laboratoire de Probabilités, Statistique et Modélisation (LPSM), Sorbonne Université, 4 Place Jussieu, 75005 Paris, France;  Email: {nicole.el$\_$karoui@sorbonne.universite.fr}}.}
    
\date{}

\renewenvironment{abstract}
 {\small
  \begin{center}
  \bfseries \abstractname\vspace{-.5em}\vspace{0pt}
  \end{center}
  \list{}{
    \setlength{\leftmargin}{0cm}%
    \setlength{\rightmargin}{\leftmargin}%
  }%
  \item\relax}
 {\endlist}
\begin{document}
\maketitle

\begin{abstract}
This paper deals with the stochastic modeling of a class of heterogeneous population   in a   random environment, called  {\em birth-death-swap}.  
In addition to  demographic events,  swap events, i.e. moves between subgroups, occur in the population.  Event intensities are  random functionals of the multi-type population. In the first part,  we show that the complexity of the problem is significantly reduced by  modeling  the  jumps measure of the population, described by a multivariate counting process. This  process  is defined as a  solution of a stochastic differential system with random coefficients, driven by  a  multivariate Poisson random measure.  The solution is obtained  under weak assumptions, 
by the thinning of a  strongly  dominating point process generated by the  same Poisson measure. This key construction relies on a general  strong comparison result, of independent interest. \\
The second part is dedicated to averaging results  when swap events are significantly more frequent than demographic events.  An important ingredient  is  the stable convergence, which is well-adapted  to   the  general random environment. The pathwise construction by domination yields  tightness results straightforwardly.  At the limit, the demographic intensity functionals are averaged against random kernels depending on swap events.  Finally, under a natural assumption, we show    the convergence of the  aggregated population to a  ``true'' birth-death process in random environment,  with non-linear intensity functionals.
\end{abstract}


{\it MSC key-words: 60G55, 60G60,60F17,91D20,92D25.
}

%---------------------------------------------------------
\section{Introduction} 
This paper is motivated by the problem of modeling stochastic heterogeneous populations, in which individuals are characterized by a finite but potentially large number of characteristics.  
Such populations  are usually modelled by multi-type Markov birth-death processes (MBDPs), where the population evolution is determined by demographic events (births and deaths), occurring at rates  depending  only on the population state.   
 The aim of this paper is to  introduce  a more general class of  population dynamics, allowing for a more realistic description of the population.

On the  one hand, we  take into account the   moves of individuals between the different subgroups, called swap events. 
In human populations,  it  is fundamental to include  changes in the geographical or socio-economic composition  of the population   for understanding the evolution  of  demographic indicators (\cite{Dowd2014}). Composition changes can also  impact  the  evolution of biological or ecological systems.
For instance,  \cite{AUGER2000} study the deterministic evolution of communities living and migrating on a  fragmented habitat (see also \cite{PICHANCOURT2006}, \cite{MOSE2012}). In biology,  horizontal gene transfers in a stochastic population model   is studied in  \cite{billiard2016effect}.  For other examples,  see  \cite{Auger2013HawkesDove}, or  in  \cite{SONG2011} for cyber risk. 

On the other hand, we include a  time varying random environment in the model. This general theoretical  framework is motivated by the need to consider complex environments in many models.  For instance in human populations,  this includes the unpredictable reduction of mortality and/or fertilty over time, varying according to  each  individual's characteristics (\cite{CRIMMINGSPRESTONCOHEN2011}, \cite{ElKaroui2018},  \cite{KAAKAI2019}), or the influence of sudden shocks such as epidemics or economic crises,  as well as macroscopic variables (economic growth,  advances in  public health, pollution, climate change...). 

As a consequence, the demographic and swap event intensities are here  assumed to be  predictable   random  functionals $\mu^\gamma(\omega,t,Z_{t^-})$ of the population state $Z_t$,  with respect to a general filtration $(\G_t)$. This general  stochastic population dynamics is  called a {\em birth-death-swap (BDS)} population, a  terminology introduced by \cite{Huber2012}  for (Markov) particle systems.\\[1mm]
\indent In contrast with stochastic individual-based models (see e.g. \cite{FournierMeleard2004}, \cite{bansaye2015stochastic}), the  representation of BDS is not centered on the population itself, but rather on the vector process  counting the occurrences of  the different types of events, called  events  counting process.  The study of the BDS   is thus  reduced to  the study of its jump measure, described as  multivariate counting process $\bolN$ on a larger state space. Then, the BDS  population $Z$  is an affine function of $\bolN$, a  viewpoint is particularly well-suited to the presence of swap events.

  A number of standard results, such as the distributional approach,  are not adaptable  to this general framework. We  thus develop a pathwise approach, where the events counting process  is  defined  as a solution of a  stochastic differential equation (SDE) driven by  a multivariate Poisson random measure, called birth-death-swap SDE. 

 Multivariate counting processes which are solutions of  SDEs  with random coefficients, are studied independently in the preliminary Section \ref{SectionPreliminary}. The existence and uniqueness of solutions  is usually obtained by controlling expectations of the counting processes,  which requires the intensity to be dominated by a (deterministic) bounded linear function ((\cite{Massoulie1998}, \cite{Garcia06spatialbirth}, \cite{bansaye2015stochastic}).   Our  pathwise approach  differs significantly.   In Proposition \ref{comparison point process},  we 
show that if two solutions of SDEs driven by the same Poisson measure  have ordered intensity functionals,  then they are strongly ordered.   Given a well-defined dominating process, 
we obtain by thinning of this process the existence and uniqueness of all solutions of SDEs with smaller intensity functionals.

The construction by strong domination is central in the paper, and allows to relax usual assumptions of sublinear growth. As an application of Proposition \ref{comparison point process}, the solution of the  BDS-SDE is  obtained  in  Theorem 
\ref{ThBDSexistence}, first of all by  proving the existence of a multivariate counting  process, strongly dominating $\bolN$. 
\\[2mm]
\indent  In the second part of the paper, the BDS system is studied in the presence of two timescales, 
 when swap events occur on a faster timescale than  demographic 
events. Such  swap events  may be  fast  changes to the social and/or geographical population composition, migrations between different patches or changes of strategies. 
In this context, the study of the aggregated  population $Z^\natural= \sum_{i=1}^p Z^i$  is important,   in order to understand macro demographic indicators or to study   population viability.

Section \ref{SectionTwoTimescales} is dedicated to proving a general averaging result for the  demographic counting process (the process counting the 
different types of demographic events) and the aggregated population.  Due to  the general random environment, averaging results such as \cite{Kurtz1992} or \cite{Yin2012continuous} cannot be applied or extended to this  framework, since event intensities are  random functionals of the population.  In Theorem  \ref{ThIdentificationResult}, we show more particularly that at the limit, the intensities of demographic events are averaged against invariant measures of pure swap processes with ``frozen environment''. 

 The construction of the  BDS-SDE by strong domination yields straightforwardly tightness results for the demographic counting processes,  without needing additional uniform integrability assumptions.   The  main difficulty lies in  defining  the right probability spaces and objects on these spaces, together with the right type of convergence. 
 We overcome this difficulty by  strongly relying on the stable convergence of the concerned processes. This mode of convergence, which extends the convergence in distribution, is  particularly well-suited for identifying limits in  random environment. 
An important point is that  the family of two-timescales  BDS processes $(Z^\epsilon)_\epsilon$ is not tight, due to the ``explosion" of swap events. The idea is to introduce the product space $(\Omega \times\mathbb R^+, d\P\times e^{-s}\d s)$ equipped with the stable convergence, and to  consider these processes as simple random variables $\tilde Z^\epsilon(\omega,s)=Z^\epsilon_s(\omega)$ on this space. \\[1mm]
\indent Finally, these results are applied in Section  \ref{SectionWeakConvergence}   to obtain the convergence in distribution of the demographic counting process, under an additional assumption on swap events, using the construction by strong domination.  The aggregated population processes converge to a ``true ''  Birth-Death process in random environment with averaged intensities. In particular, using a toy model, we show how the heterogeneity of the population can generate  non-linearities in aggregated mortality rates, which is not taken into account in standard mortality models for human populations.\\[1mm]
\indent To summarize, in  the preliminary Section \ref{SectionPreliminary}   we prove a pathwise comparison result, for multivariate counting processes solution of SDEs driven by the same Poisson measure, and  its converse. The general model  and  its pathwise representation are presented in Section \ref{SectionBDSsystem}, where we  establish the existence and uniqueness of the BDS-SDE in Theorem \ref{ThBDSexistence}. 
The last two sections are dedicated to  the averaging results in the presence of two timescales.  The two timescale BDS system is  introduced at the beginning of Section \ref{SectionTwoTimescales}, which contains the main averaging results, after a brief remainder on the stable convergence.  These results are then applied in Section \ref{SectionWeakConvergence}.
\section{Preliminary results on  multivariate counting processes}\label{SectionPreliminary}${}$\\

 \noindent Pathwise  representations of spatial (Markov) birth death processes have been  considered by several authors (\cite{FournierMeleard2004}, \cite{Garcia06spatialbirth},\cite{GarciaKurtz2008},  \cite{bansaye2015stochastic}, \cite{Bezborodov2015}), based on the  realization of the population process as solution of a SDE driven by a Poisson measure. Our point of view for  birth-death-swap (BDS) populations  is  slightly different,  since we  represent the jumps measure of the population, defined as a multivariate counting process, rather than the population process itself. \\
We start by recalling  some useful properties of multivariate counting processes, and their pathwise representation by thinning of Poisson measures, and then prove two results  (Proposition \ref{comparison point process} and its converse Proposition \ref{PropReciproqueComparaison}), on the pathwise comparison of SDEs by coupling.  These results, of independent interest, are central for proving existence and averaging results of  following sections.

\subsection{Stochastic intensities}
The  probability space, denoted by $(\Omega, \G, \P)$, is  equipped with a general filtration  $(\G_t)$, with the usual assumptions of right-continuity and completeness. The  predictable $\sigma$-field generated by the $(\G_t)$-adapted, left-continuous,  processes is denoted $\cP(\G)$.\\

A $(\G_{t})$-counting process $N$ is an adapted  $\N$-valued càdlàg process, starting from $N_0=0$ and with jumps of size 1. $N$  is  said  to have the  (predictable) $(\G_{t})$-intensity $(\lambda_t)$ if  $(\lambda_t)$ is a non-negative,  $(\G_t)$-predictable process such that for any time $t$, $\int_0^t \lambda_s \d s <\infty,$  a.s.,  and $(N_t -\int_0^t\lambda_s \, ds)$ is a $(\G_t)$-local martingale (\cite{bremaud1981point}). Equivalently,  for all  nonnegative predictable processes $H$,
\begin{equation}
\E[\int_0^{+\infty} H_s \d N_s] = \E[\int_0^{+\infty} H_s \lambda_s \d s].
\end{equation}
Any predictable support condition for $N$ can be formulated using the intensity, since for any    $\Delta \in \mathcal P(\G)$
\begin{equation}
\label{EqTransfertSupportIntensity}
\int_0^t \ind_\Delta(s)\,\d N_s=0 \quad \Longleftrightarrow \quad \ind_\Delta(s)\lambda_s = 0, \> \>\d s \times \d \P\> a.s..
\end{equation}
The filtration $(\G_t)$ plays  a critical role.  In standard models, the information considered is usually the  minimal filtration  ${\cal F}_{t}^N= \sigma(N_s; s\leq t)$ generated by the past history of the counting process. Then, the intensity may only  be  deterministic functions of the past  $[N]_{t^-} =(N_s)_{s < t}$, i.e. $\lambda_t^= f(t,[N]_{t^-})$ of the counting  process.  The canonical framework is sometimes extended to the case where $(\G_t) = (\G_0 \otimes \F_t^N)$.  Then, $\lambda_t=f^\gamma (E_0, [N]_{t^-})$, with $E_0$  a  $\G_0$-environment variable. However,  this assumption is quite restrictive since  this means that the randomness of the environment  is completely known at $t=0$. 

Two counting processes with the same $(\G_t)$-intensity may not necessarily have the same distribution. When  $(\G_t) = (\G_0 \otimes \F_t^N)$, the property is true (\cite{jacod1975multivariate}),  but not in the general case. Thus, we adopt a stochastic differential equation approach to define  properly BDS processes, based on the thinning of Poisson measures. 
\paragraph{Multivariate counting processes} Let $\J$ be a finite set of cardinal $|\J|$.  A $(\G_t)$-multivariate  counting process $N=(N^\gamma)_{\gamma \in \J}$ is a vector of $|\J|$ counting processes with no common jumps, and its $(\G_t)$-intensity is the vector $(\lambda_t) =((\lambda_t^\gamma)_{\gamma \in \J})$  of the  coordinates' intensity.

\subsection{Thinning of Poisson measures}  
Let  $Q$ be a $(\G_t)$-Poisson measure on $\mathbb R^+\times \mathbb R^+$  of Lebesgue intensity measure $\d t\times \d\theta$.  Properties of $Q$  are defined in reference to the large filtration $(\G_t)$: in particular, $(Q(]s,t+s]\times ]0,K]))_{t\geq 0}$ is a Poisson process of intensity $(\lambda_t)=(Kt)$, independent of $(\G_s)$.

For any predictable set $\Delta \in \mathcal{P}(\G)\otimes \mathcal B(\mathbb R^+)$,  the restriction of the Poisson measure to $\Delta$,
$$Q^\Delta(\d t, \d\theta) = \indi_\Delta(t,\theta)Q(\d t, \d\theta),$$ is the random counting measure of intensity ${q^\Delta(\d t,\d\theta)=\indi_\Delta(t,\theta)\d t\d\theta}$.\\
When  $\Delta= \{(t,\theta) ;\, 0< \theta \leq \lambda_t\} $ with $(\lambda_t)$  a  given predictable intensity process, the marginal
\begin{equation*}
N_t=\int_0^t \int_{\mathbb R^+} Q^\Delta(\d t,\d \theta)   =\int_0^t \int_{\mathbb R^+}\indi_{\{0<\theta\leq \lambda_s\}}Q(\d s, \d\theta)=\int_0^t Q(\d s, ]0,\lambda_s]),
\end{equation*}
is the counting process of $(\G_t)$-intensity $(\lambda_t)$. $N$ is obtained by {\em thinning and projection} of the  Poisson measure $Q$. 

When the  $\lambda_t := \alpha(t,N_{t^-})$
 is  a predictable  functional $(\alpha(\omega,t,n))$ of the process, called intensity functional, the thinned process is  solution of an SDE driven by $Q$. \\
 
The same holds for multivariate counting processes $(N^\gamma)_{\gamma \in \J}$  of $(\G_t)$-multivariate intensity functional $\alpha(\omega,t,y) =(\alpha^\gamma(\omega,t,y))_{\gamma \in \J}$ depending on $y \in \N^{|\J|}$, and  the multivariate SDE is denoted by
\begin{equation}
\label{EqMultivarSDE}
\d N_t = Q(\d t, ]0,\alpha(t,N_{t^-})]), 
\end{equation}
where  $Q=(Q^\gamma)_{\gamma \in \J}$  is a multivariate Poisson measure (the measures $Q^\gamma$ are independent), and $ Q(\d t, ]0,\alpha(t,N_{t^-})])=  (Q^\gamma(\d t, ]0,\alpha^\gamma(t,N_{t^-})]))_{\gamma \in \J}$.   Throughout this paper,  a solution  is considered well-defined if it stays finite in finite time.

The variable $\omega$ is generally omitted (as usual), but is sometimes written explicitly to recall that we work with stochastic intensities.

\subsection{Pathwise comparison of SDEs driven by the same Poisson measure} 
The comparison of  counting processes with ordered intensity  processes has been the subject of several papers (see e.g. \cite{Preston1975spatial}, \cite{bhaskaran1986almost},  \cite{rolski1991stochastic},  \cite{Bezborodov2015}),  usually based on a distributional approach. The pathwise approach presented below,  based on the  realization of counting processes using Poisson measures,  is  actually better adapted to study this problem  in a general setting, and simplifies proofs significantly.

Let us first present the main ideas of the result  in the most simple case.
Take  $N_t^{i} = \int_0^t Q(\d s, ]0,\lambda^i_s])$ be  two counting processes obtained by thinning of the same Poisson measure $Q$, and with  ordered intensities  $\lambda^1_t \leq \lambda^2_t$ a.s.  Then  
$$N_t^{1} = \int_0^t Q(\d s, ]0,\lambda^1_s \wedge \lambda^2_s]) = \int_0^t  Q^{\Delta_2}(\d s, ]0,\lambda^1_s])$$
is also  obtained by  thinning of the measure $Q^{\Delta_2}$. It follows directly that for each $t\geq 0$, $N^1_t \leq N^2_t$ a.s. Bust most importantly, $N^2- N^1$ is also a counting process, i.e. all jumps of $N^1$ are jumps of $N^2$, which is the key element of the construction.\\
The application to SDEs driven by Poisson measures is not direct, since the natural order of random intensity functionals does not necessary imply an order on the  intensities processes. However, applying similar ideas yields the general result.
\paragraph{Strong order of counting processes and intensity functionals} A multivariate counting process  $N=(N^\gamma)_{\gamma \in \J}$ is said to be  strongly dominated by $N'=(N^{'\gamma})_{\gamma \in \J}$ ($N \prec N'$) (or strongly majorized in the terminology of \cite{JacodShiryaev1987limit})  if 
\begin{center}
$N - N'$ is a multivariate counting process,
\end{center}
or equivalently if all jumps of $N$ are jumps of $N'$. 
In particular, the intensity of $N$ is bounded by the intensity of $N'$ (if they exist).\\

For  $x, y \in \N^{|\J|}$,  we write $x\leq y$  if $x^\gamma\leq y^\gamma$, $\forall \gamma \in \J$.  Two predictable  intensity functionals $\alpha =(\alpha^\gamma)_{\gamma \in \J}$ and $\beta =(\beta^\gamma)_{\gamma \in \J}$ are said to be strongly ordered, $ \alpha \leq_s \beta$ if $ \forall \; t\geq 0, \;   \gamma \in \J,$ and $ x\in \N^{|\J|}$,
\begin{equation}\label{stronglyordered}
\sup_{y \leq x} \alpha^\gamma(t,y)\leq \beta^\gamma(t,x),  \quad  \; \text{a.s.}
\end{equation}
 \paragraph{}Under this strong order on intensity functionals, the strong domination holds for multivariate counting processes solutions of SDEs driven by the same Poisson measure.
\begin{proposition}\label{comparison point process}
Let $Q = (Q^\gamma)_{\gamma \in \J}$ be a multivariate Poisson measure, and $ \beta$  an  intensity functional such that there exists a unique well-defined solution $N^\beta $ of: 
\begin{equation} \label{Dominant process}
 \d N^{\beta}_t=  Q(\d t,]0,\beta(t, N^{\beta}_{t-})]). 
\end{equation}
Then, for any intensity functional $\alpha$ such that $\alpha \leq_s  \beta$, there exists a unique solution of the equation
\begin{equation}\label{Domine process}
\d N^{\alpha}_t= {Q}\big(\d t,]0, \alpha(t, N^{\alpha}_{t-})]\big),
\end{equation}
which  is strongly dominated by 
$N^\beta$, i.e. $N^\beta - N^\alpha$ is a multivariate counting process. 
 \end{proposition}
\begin{remark}
Proposition \ref{comparison point process} can be extended straightforwardly to intensity functionals $\alpha(\omega,t, [N]^\alpha_{t^-})$ and $\alpha(\omega,t, [N]^\beta_{t^-})$  depending on the past of the counting processes. 
 In this case, the strong order of intensity functionals should be modified as follow: $\alpha\leq_s \beta$ if $\sup_{[y]_{t^-} \prec [x]_{t^-}} \alpha^\gamma(t,[y]_{t^-})\leq \beta^\gamma(t,[x]_{t^-})$ for all counting paths $[x]$, with  $[y] \prec [x]$ if all jump times of $[y]$ are jump times of $[x]$.\\
If $\alpha$ is uniformly bounded by a constant, then we  can retrieve directly  the result of  Theorem 1 in \cite{Massoulie1998},  by strong domination with a simple Poisson process. 
\end{remark}
 \begin{proof}
 We start with the one dimensional case.  Jumps times of the Poisson measure $Q$ cannot be enumerated increasingly, so the main idea is to  replace the driving Poisson measure $Q$  by the  random measure $ Q^{\Delta_\beta}(\d t;\d \theta)=\indi_{\Delta_\beta}(t,\theta) Q(\d t, \d \theta)$, with $\Delta_\beta = \{(t,\theta) ; 0 < \theta \leq  \beta(t, N^{\beta}_{t-} )\}$, whose  projection on the first coordinate  is $ N^\beta$. \\
Since $ N^\beta$ is non-exploding, the increasing sequence $(T_j)$  of its jump times  verifies $ \lim T_j = +\infty$, and
$ Q^{\Delta_\beta}$ can be characterized by the sequence $(T_j, \Theta_j )_{j\geq 1 }$, where $\Theta_j$ is the mark of $Q$ associated with $T_j$. \\
\rma {\em Existence:} We consider  a slightly different version of Equation \eqref{Domine process},  driven by $Q^{\Delta_\beta}$:
 \begin{equation}\label{Modifiedequation}
\d\tilde N^{\alpha}_t= Q^{\Delta_\beta}(\d t, ]0, \alpha(t, \tilde N^{\alpha}_{t-} )]) = Q(\d t,]0,\alpha(t, \tilde N^{\alpha}_{t-} )\wedge  \beta(t, N^{\beta}_{t-} )]),
\end{equation}
Thanks to the increasing enumeration of jump times of  $Q^{\Delta_\beta}$,
the unique solution  of (\ref{Modifiedequation}) can be built recursively as:\\[1mm]
\centerline{$\tilde{N}_t^{\alpha} = \sum_{j= 1}^\infty \indi_{\{T_j\leq t\}}  \indi_{\{\Theta_j \leq \alpha(T_j, \tilde Y^{\alpha}_{T_{j-1}})\}}$.}\\[1mm]
Since $\tilde{N}^\alpha$ is obtained by thinning of $Q^{\Delta_\beta}$, the counting process is strongly dominated by $N^\beta$ by construction. Then, since $\alpha \leq_s  \beta$, $\alpha(t, \tilde N^{\alpha}_{t} )\leq   \beta(t,  N^{\beta}_{t} )$, and thus $\tilde N^{\alpha}$ is solution of \eqref{Domine process}, which achieves to prove existence. \\
\rmc {\em Uniqueness:} It remains to prove that any solution of \eqref{Domine process} is solution of \eqref{Modifiedequation}.  Let  $ N^\alpha$ be a solution of \eqref{Domine process} and $T^\alpha_1$ its first jump time, associated with 
the mark $\Theta_1$ of $Q$. By definition of the thinning procedure, $\Theta_1 \leq \alpha(T_1^\alpha,0) $. By assumption, $\alpha(T_1^\alpha,0) \leq \beta(T_1^\alpha,0)$ and thus $T_1^\alpha$ is also a jump of $ N^\beta$. By iterating this argument, we obtain that all jump times of $ N^\alpha$ are jump times of  $ N^\beta$, or equivalently that $ N^\alpha \prec  N^\beta$. In particular, $\alpha(t,N^\alpha_{t^-}) \leq \beta(t,N^\beta_{t^-})$ a.s, and thus $N^\alpha$ is also the unique solution of \eqref{Modifiedequation}. \\
The proof can be generalized to the multivariate case, by noticing  that the vector measure $Q$ can be seen as a marked measure $\bar Q(\d t , \d \theta, \d \gamma)$,  and  $ N^\beta$ and $ N^\alpha$ as marked processes (marks being the components of the vector  which jumps).
\end{proof}
The next result is very useful since the sequence of jump times of $N^\beta$ is a localizing sequence for all local martingales $(N^\alpha_t - \int_0^t \alpha(s,N^\alpha_s) \d s)$.
\begin{lemma}
\label{CorFamilyStrongDomination}
Under the assumptions and notations of Proposition \ref{comparison point process}, let $(S_n)$ be the jumps sequence  of $N^\beta$. Then $S_n \rightarrow \infty$, and 
For all $\alpha \leq_s \beta$ and $n\geq 0$, 
$$ N^\alpha_{t\wedge S_n} - \int_0^{t\wedge S_n} \alpha(s,N^\alpha_s) \d s \text{ is a } (\G_t)-\text{martingale}.$$
\end{lemma}
\begin{proof}  Let  $|y| = \sup_{\gamma \in \J} |y^\gamma|$. Since all jump times of $N^\alpha$ are jump times of $N^\beta$, $\vert N^{\alpha}_{t \wedge S_n} \vert  \leq n$. This yields
$$\E[\sup_{0\leq s \leq t}| N^\alpha_{t\wedge S_n} - \int_0^{t\wedge S_n} \alpha(s,N^\alpha_s) \d s|]\leq 2\E[|N^\alpha_{t\wedge S_n}|]  \leq 2n,$$
which achieves the proof. 
\end{proof}
\subsection{Converse problem} The next result is in some sense a converse to  Proposition \ref{comparison point process}: 
given two multivariate counting processes  $N \prec N'$,  $N$ can be written as the solution of a thinning equation driven by a marked measure on $\mathbb R^+\times [0,1]$ with same jump times as $N'$.  Several similar results  appear in the literature. Theorem 2 of \cite{rolski1991stochastic}  is a  result close to Proposition \ref{PropReciproqueComparaison}, although from a distributional viewpoint, which yields a rather  long proof.  See also \cite{jacod79calcul,el1977representation,Massoulie1998}  for results when there is no domination assumption (the driving measure is a Poisson measure). \\[2mm]
\indent Only the one dimensional case  is presented here for conciseness, as proofs are the same in the multivariate case.  We consider two counting processes $N \prec N'$, with respective jump times $(T_n)$ and $(T_n')$, and  $(\G_t)$-predictable intensities $(\lambda_t)$ and $(\lambda'_t)$ (where $\lambda'$ is positive). The  thinning ratio is denoted by:
$$\phi_t := \dfrac{\lambda_t}{\lambda'_t}$$
Since, $N\prec N'$,  $\phi_t \leq 1$, a.s, and $N^c  = N'- N$ is a counting process, with jump times  denoted by $(T_n^{c})_{n\geq 1}$. The main idea is to accept a jump time $T'_n$ of $N'$ as a jump time of  $N$ with probability $\phi_{T'_n}$. 

We assume that there are two  independent sequences $(U_n)$ and $(V_n)$ of i.i.d uniform variables on $[0,1]$,  independent of $\G_\infty$.
Note that the probability space  can always be extended to satisfy this assumption.

 To each jump time $T_n$ (resp $T_n^{c}$) is attached the random variable $U_n$ (resp $V_n$),  defining
the random counting measures  on $\mathbb{R}^+ \times [0,1]$:
\begin{equation*}
Q^N(\d t, \d u) =\sum_{n \geq 0} \delta_{T_n} (\d t) \delta_{U_n}(\d u), \quad Q^{N^c}(\d t, \d u) =\sum_{n \geq 0} \delta_{T_n^{c}} (\d t) \delta_{V_n}(\d u).
\end{equation*}
The image  of $Q^N$ (resp $Q^{N^c}$) under the  random transformation $(t,u) \mapsto (t, \phi_t u)$ (resp$ (t,u) \mapsto (t, \phi_t + (1-\phi_t)u))$ 
  are defined by 
\begin{equation*}
\widehat Q^N(f) = \int_{\mathbb R^+ \times  [0,1]} f(s,u\phi_s) Q^N(\d s, \d u), \quad \widehat Q^{N^c}(f) = \int_{\mathbb R^+ \times  [0,1]} f(s,\phi_s + (1- \phi_s)u) Q^{N^c}(\d s, \d u).
\end{equation*}
Finally, let $\displaystyle \hat Q' = \widehat  Q^N +\widehat  Q^{N^c} := \sum_{n\in \N} \delta_{T_n'}\delta_{{U}_n'}$.
\begin{figure}[H]
\centering
\includegraphics[scale=0.4]{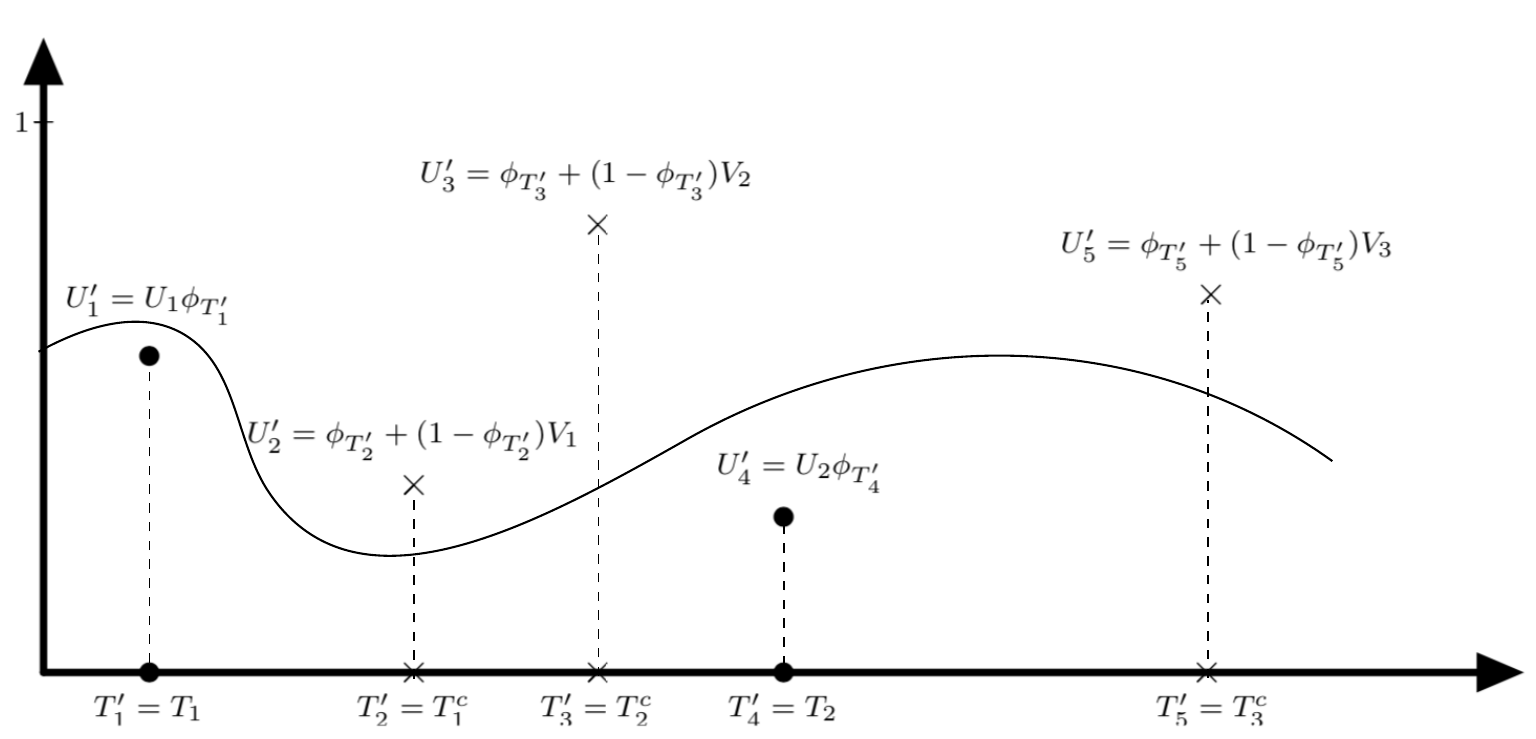}
\vspace{-0.3cm}
\caption{\textit{ Representation of $\hat Q'$.}}
\end{figure}
Proofs of Proposition \ref{PropReciproqueComparaison} and Corollary \ref{UniquenessDomination} are in Appendix \ref{AppendixProof}. 
\begin{proposition}  
\label{PropReciproqueComparaison}
\rmi Using the previous notations, $N$ is the solution of the following thinning equation driven by $\hat Q'$:
\begin{equation}
\label{eqreciproquecomparaison}
\d N_t = \hat Q'(\d t, ]0,\phi_t]) =  \hat Q'(\d t , ]0, \frac{\lambda_t}{\lambda'_t}]). 
\end{equation}
\rmii Let $\F_t ^{U,V} = \sigma(U_1, ... , U_{N_t}, V_1, ..., V_{N^c_t})$ be the $\sigma$-algebra keeping track of the  marks  attached to the  jump times before $t$. 
 Then,   $\hat Q'$ has the  $(\G_t \vee \F_t^{U,V})$-intensity   $\lambda'_s \d s \otimes \d u$.\\
 In particular,  $({U}_n')_n$ is a sequence of  i.i.d uniform  random variables on $[0,1]$, such that for all $n\geq 1$, $U_n'$ is independent of $T'_n$ and $\G_{T_n^{'-}}$.
\end{proposition}
Corollary \ref{UniquenessDomination} below  shows that a sufficient condition for two processes with the same intensity functional to have the same distribution 
 is to be  strongly dominated by the same process. 
\begin{corollary}
\label{UniquenessDomination}
Let $N'$ be a multivariate counting process of $(\G_t)$ intensity $(\lambda'_t)$, and $\alpha$ an intensity functional, defined on the initial  probability space  $(\Omega, (\G_t), \P)$. \\
Let $N^i$, $i=1,2$ be two multivariate counting processes  defined on  extensions  $((\Omega \times  \Omega_0, (\bar \G_t), \R_i)$ of  $(\Omega, (\G_t), \P)$.\\
If $N^1$ and  $N^2$ have both the intensity functional $\alpha$ and  are strongly dominated by $N'$,  then they have the same distribution. 
\end{corollary}
%
%=============================================
\section{Birth-Death-Swap systems} \label{SectionBDSsystem}${}$\\
%=============================
%=============================

\noindent We can now introduce the Birth-Death-Swap (BDS)  model. The  population is structured in  $p$ subgroups, composed of individuals sharing the same characteristics. 
Each individual is  characterized by a  number of discrete attributes such as the  living area/patch, eating habits, strategy,  level of income...  This can  imply a large number $p$ of subgroups, 
which  favors moves between subgroups. 
The BDS population process is the vector process $(Z_t)= ((Z^1_t,..,Z^p_t))$ of the number of of individuals in the different subgroups. The population composition is modified by demographic events, but also  by moves/internal migration between subgroups, called swaps events.  We assume that  two events cannot occur at the same date.

Any BDS population process can be written as a linear function of the multivariate counting process $\bolN$, counting the different types of events occuring in the population.  This  process, defined on a larger space, is the building block for defining the BDS system in Section \ref{BDSSDE}. In particular, any BDS population can be defined from an initial population and a multivariate counting process $\bolN$, given that $\bolN$ verifies a necessary and sufficient support condition.
\subsection{General Setup}
The state space of  BDS population processes is  $\N^p$,  equipped with $(\bolei)_{i=1..p}$,  where ${\bolei = (0,\cdots,0,1_{i},0)\in \N^p}$. For any  vector $x$ of arbitrary size $n$, 
the sum of its coordinates  is denoted by $x^\natural=\sum_{i=1}^n x^i.$
\subsubsection{Events description}  
A swap event from a nonempty subgroup $i$ to $j$  modifies a population $z\in \N^p$ into $z' = z-e_i+e_j$, while a  birth in $j$  yields  the transformation   $z\to z+e_j$, and a   death in $i$ yields   $z\to z-e_i$ (if $z^i \neq 0$).  \\
Births (resp. deaths) may be considered as swap events from (resp. to) a fictitious  subgroup labeled by $b$ (resp. $d$),  with  $e_{b}= e_{d}=(0,\cdots ,0)$. 

The family of events is indexed by the  set $\J = \J^s \cup  \J^\dem $ where $\J^s$ is the set of the $p(p-1)$ swap events types, and $\J^\dem =\J^b \cup \J^d$  is the set of the $2p$ demographic (birth and death) events types:
\begin{equation*}
\J^s =\{\kappa= (i, j);  1 \leq i,j \leq p , i\neq j\},\quad \J^b=\{(b, j) ;\,  1 \leq j \leq p\},\quad \J^d = \{(i, d) ;\, 1 \leq i \leq p\}. 
\end{equation*}
In short, the impact on the population of an event  of type $\gamma=(\alpha,\beta) \in  \J$   is:
\begin{equation} \label{phi}
\phi(\alpha,\beta)= e_{\beta}- e_{\alpha} \in \mathbb Z^p.
\end{equation}
\subsubsection{Space of counting vectors} \label{ssb:Eventnotation} 
We are interested in counting  events  occurring in the population.  The new state space, called  space of counting vectors, is   the space of $\N^{p(p+1)}$ vectors $\bolnu=(\nu^\gamma)_{\gamma \in \J}$
 indexed by $\J$. When  swap events are distinguished from demographic events, we write $\bolnu =(\bolnu^s, \bolnu^{\dem})$, where $\bolnu^s=(\nu^{\kappa})_{\kappa \in \J^s}$  and $\bolnu^{\dem} = (\nu^\gamma)_{\gamma \in \J^\dem}=(\bolnu^{b},\bolnu^d)$, with  $\bolnu^b = (\nu^\gamma)_{\gamma \in \J^b}$, $\bolnu^d = (\nu^\gamma)_{\gamma \in \J^d}$. The same notation may be used for $\phi$. The operator $\odot$ is defined by:
 \begin{equation}
  \label{matrixproduct}
 \begin{cases}
 \phi  \odot \bolnu=  \sum_{\gamma\in \J} \phi(\gamma)\nu^\gamma,\\
\phi \odot \bolnu=\phi^s \odot \bolnu^s +\bolnu^{b}-\bolnu^d.
\end{cases}
 \end{equation}
\subsubsection{Events counting process of the BDS population} 
Let $Z$ be a BDS population process. 
If  an event of type $\gamma \in \J$ occurs  at time $t$ then $\Delta Z_t=Z_t-Z_{t-} = \phi(\gamma)$, and thus
\begin{equation*}
Z_t - Z_0= \sum\limits_{0<s \leq t}\Delta Z_s = \sum\limits_{0<s \leq t} 
\sum\limits_{\gamma\in \J} \indi_{\{ \Delta Z_s=\phi(\gamma)\}}\,\phi(\gamma)
%= \sum\limits_{\gamma\in \J} \phi(\gamma)\sum\limits_{0<s \leq t}\indi_{\{ \Delta Z_s=\phi(\gamma)\}}.
\end{equation*}
The previous equation can be rewritten as 
\begin{eqnarray}
 \label{jump counting} 
 Z_t= Z_0+\sum\limits_{\gamma\in \J}\phi(\gamma)N^\gamma_t, \quad \quad 
N^\gamma_t = \sum_{0 < s\leq t} \indi_{\{\Delta Z_s=\phi(\gamma)\}},
\end{eqnarray}
where  $N^\gamma$ counts the number of events of type $\gamma$ that occurred in the population.\\
The multivariate counting process $\bolN=(N^\gamma)_{\gamma \in \J}$, called the {\em events counting process}. Taking values in $\N^{|\J|} = \N^{p(p+1)}$, this process is related to $Z$ by the affine relation, 
 \begin{equation}
 \label{EqAffine}
  Z = Z_0 +  \phi \odot \bolN.
 \end{equation}
 An equivalent point of view is to consider the jumps measure of the population process, 
$ J(\d t,\d\gamma)=\sum\limits_{s>0 } \sum\limits_{\gamma'\in \J} \indi_{\{ \Delta Z_s=\phi(\gamma')\}}\delta_{\gamma'}(\d\gamma) \delta_s(\d t)$ . \\

The evolution of the total size of the population, called the {\em aggregated process}, is 
 $Z^\natural= \sum_{i=1}^p Z^i = Z_0^\natural + N^{b, \natural}-N^{d,\natural}.$ 
Note that $Z^\natural$ only depends on demographic events since swap events don't change the size of the population.
\subsection{Birth-death-swap (BDS) SDE} 
\label{BDSSDE}
By the affine relation \eqref{EqAffine}, a BDS population $Z$ can thus be defined from a multivariate counting process $N$. Our main assumption concerns the  $(\G_t)$-intensity of the different events occurring in the population, assumed to be  depending on the population via a multivariate $(\G_t)$-predictable  functional $\bolmu( t,Z_{t^-})=\bolmu(t,Z_0+\phi \odot N_{t^-})$. An additional support condition on $\mu$ is needed to ensure that no death or swap events  can occur in an empty subgroup.   \\
Results of this Section can be extended straightforwardly to include dependence in the past of the population process.   

\begin{definition}[BDS intensity functional]
\label{DefBDSintensity}
A  Birth-death-swap intensity functional $\bolmu(t,z)=(\mu^\gamma(t, z))_{\gamma \in \J}$ is a non-negative vector of $(\G_t)$-predictable  functionals depending on  $z\in \N^p$,  satisfying the following support condition:
  \begin{equation} \label{pop intensity condition}
 \mu^{(i,\beta)}(t,z)\ind_{\{z^i=0\}}\equiv 0,  \quad \d t\otimes \d\P\,a.s, \; \forall  \; (i,\beta) \in \J^d\cup \J^s. 
\end{equation}
\end{definition}
By \eqref{EqTransfertSupportIntensity},  if $\bolmu$ is a BDS intensity functional  then Assumption \eqref{pop intensity condition} is equivalent to the following support  constraint  for $\bolN$:
\begin{equation*}
\forall (i,\beta) \in \J^d \cup \J^s, \quad \int_0^t \indi_{\{Z^i_{s-}=0\}}dN^{i, \beta}_s=0,\quad  \P\text{-a.s.}. 
 \end{equation*}
This  ensures   that the process $Z= Z_0 + \phi \odot \bolN$   is a well defined population, i.e. $Z^i \geq 0$, for each subgroup $i$. The BDS system $(Z_0, N,Z)$ is then  defined  from an events counting process $N$, solution of the  SDE below.
\begin{definition}[BDS system]\label{DefBDSSystem}
A BDS system of  (BDS) intensity functional $\bolmu$ and driven by $(\G_t)$- Poisson measure $\bolQ=(Q^\gamma)_{\gamma\in \J}$  is a triplet $(Z_0,\bolN, Z)$ with  $Z_0 \in \G_0$, such that  $\bolN$ is a well-defined 
solution of the multivariate SDE:
\begin{equation}
\label{ThinningGeneral}
\d \bolN_t= \bolQ(\d t,]0,\bolmu(t, Z_0 + \phi \odot \bolN_{t^-})]), \quad \text{and }   Z_t = Z_0 + \phi \odot \bolN_{t}.
\end{equation}
$Z$ is a well-defined BDS population process  of initial population $Z_0$ and events counting process $\bolN$.
\end{definition}
We draw the reader's attention to the fact that   the stochastic intensity functional $\mu$ is not assumed to be independent of  the Poisson measure $Q$. For instance, $\mu$ can be defined as  $\mu(t,z) =f(Y_{t^-},z)$,  where $(Y_t)$ is also the solution of an SDE driven by $Q$. \\
In particular, two BDS systems with the  same intensity functional can have different distributions. 
For instance, if   $\mu(t,z) =f(Y_{t^-},z)$ with  $(Y_t)$ the solution of an SDE driven by $Q$,  and if $Q'$ is a multivariate Poisson measure independent of $Q$, then the BDS system of intensity functional $\mu$ and driven by $Q$ does not have the same distribution than the BDS system of intensity functional $\mu$,  driven by $Q'$.

\paragraph{Examples} The description of BDS populations differs  from that of  Individual-based models. Here, the population is observed at an intermediary scale.  
In particular, $\mu^\gamma(t,Z_{t^-})$ is the intensity corresponding to the occurrence of the event of type $\gamma$ in {\em all} the population, not to be confused with the  individual rate at which an event an occur to  an individual.  For instance,  $\mu^{(i,d)}(t,Z_{t^-})\d t$ is  the expected number of deaths in subgroup $i$  occurring in the small  interval $]t,t+dt]$, conditionally to $\G_{t^-}$.\\[1mm]
\rma {\em Linear intensities} When the functionals  $\mu^\gamma$ depend linearly on the population,  a direct interpretation in terms of individual rates can be given. For instance, when  $\mu^{(i,d)}( t,z) = d^i_tz^i$, all individuals in subgroup $i$  have the stochastic death rate $d^i_t$.  If $\mu^{(i,j)}(t,z) =k^{ij}_t z^i$, an individual in subgroup $i$ can move to subgroup $j$ with stochastic rate $k^{ij}_t$. \\
When the intensity of a birth event $(b,j)$ is $\mu^{(b,j)}(t,z) = b^j_t z^j$, individuals in subgroup $j$ give birth to an individual of same characteristics at rate $b^j_t$. Mutations can be included, by taking $\mu^{(b,j)}(t,z) = \sum_{i=1}^p b^{i}_t  m_t( i,j) z^i$, where $m_t(i,j)$ is the random probability for a individual born at time $t$ from a parent in subgroup $i$ to be in subgroup $j$. A stochastic intensity $\lambda_t$ can also be added, in order to model the entry of immigrants at rate $\lambda_t$.\\[1mm]
\noindent \rmb {\em Deterministic intensity function}  If $(Z_0,\bolN, Z)$ is a BDS system, for any $f\in \cal C_b(\N^p)$,  $f(Z_t) = f(Z_0) + \sum_{\gamma \in \J} \int_0^t (f(Z_{s^-} + \phi(\gamma))-f(Z_{s^-}))\d N_s^\gamma$, and
\begin{equation*}
f(Z_t) - f(Z_0) - \sum_{\gamma \in \J} \int_0^t  (f(Z_s + \phi(\gamma))-f(Z_s))\mu^\gamma(s,Z_s)\d s \text{ is a }(\G_t)\text{-local martingale}.
\end{equation*}
When $\bolmu$ is a deterministic function $\bolmu(z)$,  $Z$ is  a Continuous Time Markov Chain (CTMC),  solution of the classical martingale problem.

\subsection{Solutions of  the Birth-Death-Swap SDE by strong domination} 
The existence of a  well-defined  solution to the BDS-SDE \eqref{ThinningGeneral}  imposes constraints on the coefficients of the SDE.  
In many papers, intensity functionals are assumed to be uniformly bounded  or to have sublinear growth  (\cite{Massoulie1998}), \cite{Garcia06spatialbirth}, \cite{bansaye2015stochastic}, \cite{Bezborodov2015}). The latter assumption is sometimes known in the one dimensional case as the Jacobsen condition  (\cite{jacobsen2012statistical}). The sublinear growth assumption is critical since proofs are based on controlling expectations of the counting processes,  which requires the intensity to be dominated by a (deterministic) linear function.

Our approach differs, since  the existence and uniqueness of the BDS-SDE is obtained by strong domination, by applying  Proposition \ref{comparison point process}. This pathwise approach by strong domination is central in the remainder of the paper, and allows us to relax usual assumptions of sublinear growth. By Proposition \ref{comparison point process}, the proof is reduced to proving the existence of a dominating process $G$, whose intensity functional dominates $\bolmu$ (in the sense of \eqref{stronglyordered}). 

\subsubsection{Birth intensity functional domination assumption}
The birth intensity functional has to  controlled, in order  for the population size to stay finite  in finite time. 
A famous non-explosion condition  for  one-dimensional   Markov Birth  processes is the Feller condition for non linear birth intensity functions $K g$, where  $\sum1/g(n) =+\infty$.  The generalization to the multivariate case (see e.g. Proposition 10.21 in \cite{kallenberg2006foundations}) is straightforward when the birth intensity functions $(Kg_j)_{j=1..p}$ only depends of the size of the population $z^\natural=\sum_{i=1}^p z_i$, the Feller condition being applied to  $\sum_{j=1}^p g^j(z^\natural)$:
 \begin{equation}\label{Feller}
 \sum_{z=1}^\infty \frac{1}{\sum_{j=1}^p g^j(z^\natural)} =+\infty 
 \end{equation}
We show below that  existence  and uniqueness of the BDS-SDE solution  can be obtained when birth intensity functionals are dominated by such a function.  The constant $K$  can actually be replaced by a locally bounded process $(k_t)$.    In the remainder of the paper,  birth intensity functionals are assumed  to be dominated by a  separable random functional:
\begin{hyp}[Dominating Assumption]\label{HypCoxBirthDom}  
Let $(g_j)_{j=1..p}$  a vector of non decreasing functions satisfying the  condition \eqref{Feller}, and $(k_t )$ is a predictable locally bounded process. We assume that:  
\begin{equation}
\label{DomBirth}
\forall \; 1\leq j \leq p, \quad \quad \mu^{(b,j)}(t, z)\leq k_t \, g_{j}( z^\natural ) \quad  \text{a.s.}.
\end{equation}
\end{hyp}
For an event of type $\gamma$, the maximum  of $\mu^\gamma$ over the finite space of populations of size smaller than $n\in \N$ is the increasing function denoted by 
\begin{equation}\label{maxsize}
\hat \mu^{\gamma}(t, n)=\sup_{\{z^\natural \leq n\}}\mu^{\gamma}(t, z).
\end{equation}
\subsubsection{Dominating counting processes}   In order to apply Proposition \ref{comparison point process}, the first step  is to  introduce a dominating process $\bolG$, solution of an SDE driven by the same Poisson measure $\bolQ= (Q^\gamma)_{\gamma \in \J}$ than $\bolN$, and whose intensity functional strongly dominates the BDS intensity functional $\bolmu$.\\
First, the  process $G^b=(G^{\gamma})_{\gamma \in \J^b}$ of  components indexed by birth event types is defined as the solution of an SDE  driven by $\bolQ^b= (Q^\gamma)_{\gamma \in \J^b}$. Then,  $G^d=(G^{\gamma})_{\gamma \in \J^d}$ (resp. $G^s=(G^{\gamma})_{\gamma \in \J^s}$) is defined as the solution of a simple thinning equation driven by $Q^d =(Q^\gamma)_{\gamma\in \J^d}$ (resp. $Q^s=(Q^\gamma)_{\gamma\in \J^s}$).

\begin{proposition}\label{Birth event process}
Let $g=(g_j)_{j=1..p}$  a non decreasing functions satisfying the  condition \eqref{Feller}, and $(k_t )$ is a predictable locally bounded process. 
 There exists a unique well-defined solution to the   SDE,
\begin{equation}
\label{Birthdominating}
\d \bolG^b_t= \bolQ^b(\d t,]0,k_tg(t,Z_0+\bolG^{b,\natural}_{t-})]).
\end{equation}
$\bolG^b$ can be extended into a $p(p+1)$ multivariate counting process $\bolG= (G^b,G^d,G^s)$ as follows: 
\begin{equation}\label{GdeathSwapevent}
 \d\bolG^d_t =\bolQ^d(\d t, ]0,\hat \mu^d(t, Z_0^\natural + G^{b,\natural}_{t^-}) ], \quad
 \d\bolG^s_t =\bolQ^s(\d t, ]0,\hat \mu^s(t, Z_0^\natural + G^{b,\natural}_{t^-}) ],
 \end{equation}
 with $\hat \mu^e=(\mu^\gamma)_{\gamma \in \J^e}$, for $e=d,s$. 
 \end{proposition} 
 \noindent  Note that $\bolG$ is not the events counting process of a well-defined BDS  population, since the  support condition \eqref{pop intensity condition} is not satisfied.

Proposition \ref{GdeathSwapevent} yields directly the existence and uniqueness of the BDS-SDE. 
\begin{theorem} \label{ThBDSexistence} 
Let $\bolmu$ be a BDS intensity functional  verifying Assumption \ref{HypCoxBirthDom}.\\
\rmi  There exists a unique  events  counting process $\bolN$,  solution  of Equation \eqref{ThinningGeneral}:
\begin{equation*}
\d\bolN_t= \bolQ(dt,]0,\bolmu(t, Z_0 + \phi \odot \bolN_{t-}]), \quad  Z_t = Z_0 + \phi \odot \bolN_{t},
\end{equation*}
defining a BDS system $(Z_0,\bolN, Z)$ of BDS intensity functional $\bolmu$ and driving measure $\bolQ$.\\ 
\rmii  $\bolN$ is strongly dominated by the multivariate counting process $\bolG$: $\bolN \prec \bolG$. 
\end{theorem}
\begin{remark} No pathwise comparison results can be obtained straightforwardly for the BDS population  itself, especially due to the presence of the various swap events.  
The aggregated population $Z^{\natural}=\sum_{i=1}^p Z^{i}$ does not depend on swap events,  but is only dominated at each time $t$  by $Z_0^\natural +   G_t^{b,\natural}$.
\end{remark}
%%%%
\begin{proof}[Proof of Theorem \ref{ThBDSexistence}]
By Proposition \ref{comparison point process} and Proposition \ref{Birth event process}, the existence and uniqueness of the BDS SDE, as well as the  strong domination,  is reduced to showing  the strong domination of the BDS intensity functional  $\bolmu$ by that of $ \bolG$, which we denote  $\beta$.\\
$\bolmu$ can  be written as a functional $\alpha(t,\bolnu)=\bolmu(t, Z_0+ \phi\odot \bolnu)$ on the space of counting vectors $\bolnu=(\nu^\gamma)_{\gamma \in \J }$. 
By Assumption \ref{HypCoxBirthDom},  for all $(b,j) \in \J^b$ and counting vectors $\bolnu_1 \leq \bolnu_2 $:
\begin{align*}
& \alpha^{(b,j)}(t,\bolnu_1)  \leq k_t g_j((Z_0 +\phi\odot \bolnu_1)^\natural)\\
&  = k_t g_j(Z_0^\natural + \nu^{b,\natural}_1 -\nu^{d,\natural}_1)  \leq k_t g_j(Z_0^\natural + \nu^{b,\natural}_2) = \beta^{(b,j)}(t,\nu_2), \text{ a.s.}
\end{align*}
For $\gamma \in \J^s \cup \J^d$, 
 \begin{equation*}
\alpha^{\gamma}(t,\bolnu_1)=\mu^{\gamma}(t, Z_0+ \phi\odot \bolnu_1)\leq \hat \mu^{\gamma}(t, Z_0^\natural + \nu_2^{b,\natural})= \beta^\gamma(t, \nu_2), \text{ a.s.} 
 \end{equation*}
Thus, $\alpha \leq_s \beta$, which concludes the proof. 
\end{proof}
\begin{proof}[Proof of Proposition \ref{Birth event process}]
We need to show the existence and uniqueness of $\bolG^b$, the  solution of Equation \eqref{Birth event process}.  \\
Since $(k_t)$ is locally bounded,  there exists an increasing sequence of constants $(K_n)$ and of stopping times $(\tau_n)$, such that $\lim_{n\to +\infty} \tau_n = +\infty$ and for all $t \geq 0$, $ k_{t\wedge \tau_n} \leq K_n$, a.s.\\
- By Proposition \ref{comparison point process}, showing the result   with $ k_t \equiv K_n$  is sufficient to deduce the existence and uniqueness of the solution $G^{b,n}$ of the SDE  associated with $ k_t^n  := k_t \wedge K_n$, which coincides with $ G^b$ on $[0,\tau_n[$. The  proposition is then proved  by letting $n \to \infty$. \\[1mm]
 - For $n\geq 0$, a multivariate counting process with intensity {\em function} $ f^n(\bolnu^b) := (K_ng_j(\nu^{b,\natural}))_j$ is a classical multi-type Markov pure birth process, whose intensity only depends on the global size of the population.   The nonexplosion of these processes is well-known under \eqref{Feller}. It remains to realize such a process as the solution of an SDE driven by $\bolQ^b$, which can be done by pasting of stopped Poisson processes:\\
We start by introducing the doubly stochastic Poisson  process $Q^0_t = \bolQ^b_t(]0, f^n(Z_0^\natural)])$ and set $G^{b,n}_{t^-} =  0$ on $\{t \; ;  \; \bolQ_{t^-}^0=0\} $.\\
 The second step is to define  $\bolQ^1_t = \bolQ^b(]0, f^{n}(Z_0^\natural+1)])$,  and set  
 $$\d G^b_t = \ind_{\{ Q^0_{t^-}=0\}}\d Q^0_t+\ind_{\{Q^0_{t^-}>0\}}\d Q^1_t$$ on  $\{Q_{t^-}^0=0\}  \cup( \{ \bolQ_{t^-}^0 >0\} \cap \{ \bolQ_{t^-}^1 =0\}) $. By iterating this procedure, we  build a unique solution to \eqref{Birth event process}.
\end{proof}
%
%%SSSSSSSSSSSSSSSSSSSSSSSSSSSSSSSSSSSSSSSSSSSSSSSSSSS
\section{Averaging result in presence of two timescales}\label{SectionTwoTimescales}${}$\\
%SSSSSSSSSSSSSSSSSSSSSSSSSSSSSSSSSSSSSSSSSSSSSSSS
%---------------------------------------------------------------------------------

\noindent The study of the aggregated  population $Z^\natural= \sum_{i=1}^p Z^i$  is interesting  in order to understand specific features of the population from a macro viewpoint (aggregated demographic rates, viability...).  However
$Z^\natural$ is not a ``true'' birth-death process, since 
the aggregated birth (resp. death) intensity functionals $\mu^{b,\natural} = \sum_{j=1}^p \mu^{b,j}$ (resp. $\mu^{d,\natural}$) depend on the whole structure of the population $Z$, whose  composition is not constant between two demographic events,  due to swap events. The presence of a random environment also adds complexity.  The aggregated dynamics can be approximated by a reduced system in the presence of a separation of timescales, when swap events happen on a short timescale in comparison with the demographic timescale.  Then, swap events  have an {\em averaging effect} on demographic intensities, allowing for the demographic event counting process and the aggregated population to be approximated by  simpler nonlinear dynamics.
%
 %%%%%%%%%%%%%%%%%%%%%%%%%%%%%%%%%%%%%%%%%%%%%%%%%%%%
\subsection{Two timescale BDS system}\label{SubSecTwoTimescaleBDS}
%%%%%%%%%%%%%%%%%%%%%%%%%%%%%%%%%%%%%%%%%%%%%%%%%%%%%
Let  $(\bolmu^\dem, \bolmu^{s})$ be a BDS intensity functional as defined in Definition \ref{DefBDSintensity}, verifying   the domination  Assumption  \ref{HypCoxBirthDom}, and  $\bolQ = (Q_{\gamma})_{\gamma \in \J}$ a  multivariate Poisson measure.

The two timescale BDS system  $(Z_0,\bolN^\epsilon, Z^\epsilon)$  is  a BDS system driven by $\bolQ$,  of  intensity functional $\bolmu^{\epsilon} =(\bolmu^\dem,\frac{1}{\epsilon} \bolmu^{s})$ depending  on a small parameter $\epsilon$.  Obviously, $\bolmu^\epsilon$ also verifies Assumption  \ref{HypCoxBirthDom}.   The  swap events counting process  $\bolN^{s,\epsilon}$, of intensity functional  $\frac{1}{\epsilon} \bolmu^s(t,z)$,  explodes when $\epsilon \rightarrow 0$. Conversely,  the intensity functional of the demographic counting process $\bolN^{\dem,\epsilon}$ does not depend on $\epsilon$, but  its intensity depends on $Z^\epsilon$.
\begin{figure}[h]
\centering
\begin{tikzpicture}[x=2.4cm]
\draw[->,thick,>=latex]
  (0,0) -- (6,0) node[below right] {$t$};

\draw[thick] 
    (0,-4pt) -- ++(0,8pt) node[above] {$0$};
\draw[red, thick] 
    (0.2,-4pt) -- ++(0,8pt) ;
\draw[red, thick] 
    (0.3,-4pt) -- ++(0,8pt) ;
\draw[red, thick] 
    (0.5,-4pt) -- ++(0,8pt) ;
\draw[red, thick] 
    (0.6,-4pt) -- ++(0,8pt) ;
\draw[red, thick] 
    (0.7,-4pt) -- ++(0,8pt) ;
\draw[red, thick] 
    (1,-4pt) -- ++(0,8pt) ;     
\draw[red, thick] 
    (1.2,-4pt) -- ++(0,8pt) ;  
\draw[blue, thick] 
    (1.5,-4pt) -- ++(0,8pt) node[above] {$T_1^d$};
\draw[red, thick] 
    (1.7,-4pt) -- ++(0,8pt) ;
\draw[red, thick] 
    (1.85,-4pt) -- ++(0,8pt) ;
\draw[red, thick] 
    (2,-4pt) -- ++(0,8pt) ;
\draw[red, thick] 
    (2.1,-4pt) -- ++(0,8pt) ;
\draw[red, thick] 
    (2.3,-4pt) -- ++(0,8pt) ;     
\draw[blue, thick] 
    (2.6,-4pt) -- ++(0,8pt) node[above] {$T_2^d$};
\draw[red, thick] 
    (2.8,-4pt) -- ++(0,8pt) ;
\draw[red, thick] 
    (2.9,-4pt) -- ++(0,8pt) ;
\draw[red, thick] 
    (3.15,-4pt) -- ++(0,8pt) ;
\draw[red, thick] 
    (3.37,-4pt) -- ++(0,8pt) ;
\draw[red, thick] 
    (3.45,-4pt) -- ++(0,8pt) ;  
\draw[red, thick] 
    (3.8,-4pt) -- ++(0,8pt) ;
\draw[red, thick] 
    (4.12,-4pt) -- ++(0,8pt) ;
\draw[red, thick] 
    (4.23,-4pt) -- ++(0,8pt) ;
\draw[red, thick] 
    (4.37,-4pt) -- ++(0,8pt) ;
\draw[red, thick] 
    (4.6,-4pt) -- ++(0,8pt) ;  
\draw[blue, thick] 
    (4.8,-4pt) -- ++(0,8pt) node[above] {$T_3^d$};
\end{tikzpicture}
\vspace{-0.3cm}
\caption{Example of distribution of {\color{red} swap} events and {\color{blue} demographic} events}\label{graphseparationechelletemp}
\end{figure}
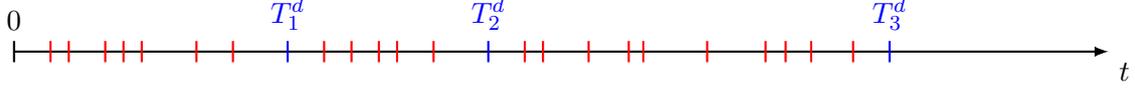

Notations introduced in Section \ref{matrixproduct} allow us to distinguish swap events from demographic events:
\begin{align}
\nonumber
& Z^\epsilon_t = Z_0 + \phi^s \odot \bolN^{s,\epsilon}_t + \bolN^{b,\epsilon}_t - \bolN^{d,\epsilon}_t, \quad \bolN^{\dem,\epsilon}_t=(\bolN^{b,\epsilon}_t, \bolN^{d,\epsilon}_t)\\
\label{EqTwotimescaleBDS}
&\d \bolN^{s,\epsilon}_t = \bolQ^s(\d t,]0,\frac{1}{\epsilon} \bolmu^s(t,Z_{t^-}^{\epsilon})]),  \quad \d \bolN^{\dem,\epsilon}_t = \bolQ^\dem(\d t,]0,\bolmu^{\dem}(t,Z_{t^-}^{\epsilon})]).
\end{align}
An important point is that since the demographic intensity functional is not modified, the dominating process  $\bolG^\dem = (\bolG^b,\bolG^d)$  defined  in Proposition \ref{Birth event process} is still a dominating process for $\Ndemeps$. 
Hence, the demographic counting processes $(\Ndemeps)_{\epsilon >0}$ are {\em  ``uniformly in $\epsilon$''  strongly dominated by $G^\dem$}. 

\paragraph*{} The aim of this section is   the study of the convergence of the demographic  counting processes $\bolN^{\dem, \epsilon}$ when swap events become instantaneous  with respect to  the demographic timescale ($\epsilon \rightarrow 0$). The aggregated population $ Z^{\epsilon,\natural} = \sum_{i=1}^p Z^{\epsilon,i}$ can also be written as a  function  $Z^{\epsilon,\natural} = Z_0^\natural + N^{b,\epsilon,\natural} - N^{d,\epsilon,\natural}$ of $\bolN^{\dem,\epsilon}$, and thus all limit results for the aggregated population can be derived from the study of $(\Ndemeps)$. 

Thanks to construction  by strong domination with $G^\dem$,  tightness properties  on the space of counting paths are obtained straightforwardly, without need of additional integrability conditions.  
This section is thus  mainly dedicated to identifying properties of   limits of $(\Ndemeps)$.  The main difficulty stems from the general random environment setting, since event intensities are random functionals  not regular  in  $t$, and do not characterize the demographic processes' distribution.  Thus, averaging results such \cite{Kurtz1992} or \cite{Yin2012continuous} does not apply here. 

In order to overcome this difficulty, we rely  on the stable convergence rather than standard tools of the convergence in distribution. This mode of convergence allows for the  randomness of the  probability space to be taken into account, and  preserves pathwise properties such as the strong domination,  which is particularly suited to our framework. However,  the right  probability spaces have to  be specified for the demographic counting processes, but also  for the family of BDS population $(Z^\epsilon)_\epsilon$.\\
Due to the explosion of swap events,  $(Z^\epsilon)_\epsilon$ cannot be considered as $\N^p$-valued processes (the family is not tight). The viewpoint introduced in \cite{Kurtz1992}, based on the convergence in distribution of the  population occupation measure,  is  not adapted either for the reasons mentioned above  (see also Remark \ref{RemarqueKurtz}). The idea here is to consider  populations $Z^\epsilon$ as   $\N^p$-random variables $\tilde{Z}^\epsilon$ on $\Omega \times \mathbb R^+$. This construction is detailed in Section \ref{SectionPopVar}.  Another technical difficulty comes from defining the stable convergence of the demographic counting processes and the population variables on the same probability space, which is done in Section \ref{SectionJointcvg}. \\[1mm]

\indent We start  by giving a short overview of this mode of convergence. 
%%%%%%%%%%%%%%%%%%%%%%%%%%%%%%%%%
 %===================================================================================
 \subsection{Overview on the stable convergence and space of rules}
  \label{SectionStblCvg}
  %===============================================================================
%
 Originated by Alfred R\'enyi, the notion of stable convergence,  is used in many limit theorems in random environment, or also for  counting processes (see e.g. \cite{Brown1981} or \cite{Jacod1987convergence}). References may be found in \cite{JacodMemin1981},  \cite{AldousEagleson1978},  or in a more recent detailed presentation in the book \cite{Hausler2015stable}. See also  \cite{castaing2004young} for a complete presentation from a topological viewpoint.
\subsubsection{Enlarged space and space of rules}
A natural extension preserving the initial  probabilistic structure  $(\Omega, \G, \P)$ is  the enlarged product space $(\bar \Omega, \bar \G) = (\Omega\times\cX, \G \otimes \mathcal B(\cX))$,  where the   identity application on $\cX$ is denoted by $ {\cal I}_d( \omega, \chi)=\chi$.

Admissible probability measures on $(\bar \Omega, \bar \G)$,  called  {\bf rules}, are probability measures whose marginal on $(\Omega,\G)$ is the given probability measure $\P$. The set of rules on $(\bar \Omega, \bar \G)$ is denoted by $\cR(\P,\cX)$.  Since $\cX$ is a Polish space, any rule $\R \in \cR(\P,\cX)$ can be disintegrated into a transition probability kernel  $\Gamma(\omega,\d \chi)$ from $(\Omega,\G)$ to $(\cX,\mathcal B(\cX))$, $\R(\d \omega, \d \chi)= \P(\d\omega)\Gamma(\omega,\d \chi)$, with 
\begin{equation}\label{KernelRule}
\R[H({\cal I}_d)]=\int_{\bar \Omega} \R(\d \omega, \d \chi)\,H(\omega,\chi) =  \int_{\Omega} \P(\d \omega)\left( \int_{\cX} H(\omega,\chi)\Gamma(\omega, \d \chi)\right)=\E[\Gamma(H)].
\end{equation}
 An example is the rule $\R^Y(\d \omega, \d \chi) = \P(\d \omega) \delta_{Y(\omega)}(\d \chi)$ associated with an $\cX$-valued $\G$-random variable $Y$. For this rule, $\Gamma^Y(\omega,\d \chi)=\delta_{Y(\omega)}(\d \chi)$ and
\begin{equation}
\R^Y[H(\cal I_d)]=\int_{\Omega} \P(\d\omega) H(\omega,Y(\omega))=\E[ H(Y)],
\end{equation}
for any $\bar \G$-r.v. $H(\omega,\chi)$.
In particular, the $\R^Y$-probability distribution of  the canonical variable $ {\cal I}_d$ is  the probability distribution $\mu^Y$ of $Y$ under $\P$. Thus, the marginal  of $\R^Y$ on $\cX$ is $\R_\cX^Y =\mu^Y$, and the rule can be seen as a coupling between $\P$ and the distribution of $Y$.
%%%%%%%%%%%%%%%%
%
\subsubsection{Stable convergence}
The different modes of convergence of probability measures  are usually characterized by their family of tests functions. For example, the convergence in distribution is described by   bounded continuous functions on $\cX$, here denoted by $\cal C_{bc}(\cX)$, ($\cC_{b}(\cX)$ for  bounded functions). The {\em stable convergence}  is an extension of the  convergence  in distribution to the space of rules. 
 The class of test functions is extended to the family $\cC_{bmc}(\Omega\times\cX)$, of  bounded functionals $H(\omega, \chi)$, continuous in $\chi$ for any $\omega$, but only  measurable in $\omega$.%
\begin{definition}[Stable convergence, first viewpoint] 
\label{stablecv} 
A sequence of $\G$-random variables $(Y_n)$ converges stably to a rule $\R = \P\Gamma \in \cR(\P,\cX)$ iff
\begin{equation}
\forall H \in  \cC_{bmc}(\Omega\times\cX), \quad   \E[H( Y_n)] \rightarrow_n \E[\Gamma(H)].
\end{equation}
\end{definition}
%€€€€€€€€€€€€€€€€€€€€€€€€€€€€€€€
Thus, the stable convergence of $(Y_n)$  can be considered as the convergence of random functionals of these variables, towards a  random kernel defined on the initial probability space. 

A second  viewpoint is to define the stable limit as a random variable defined on the enlarged space $(\bar{\Omega}, \bar{\G},\R)$ endowed with the limit rule $\R$. The variables $Y_n$  are naturally extended to  $(\bar{\Omega}, \bar{\G},\R)$, by setting $Y_n(\omega,\chi) =Y_n(\omega)$.
\begin{definition}[Stable convergence, second  viewpoint]
\label{StableCvgScdViewpoint}
A sequence of $\G$-random variables  $(Y_n)$ converges stably to $\R$ iff,
\begin{equation}
 \forall H \in \mathcal C_{bmc}(\Omega\times \cX),\quad \R[H(Y_n)]= \E[H(Y_n)]\rightarrow \R[ H({\cal I}_d)],
\end{equation}
since $\R[ H({\cal I}_d)] = \E[\Gamma(H)]$ by \eqref{KernelRule}.
\end{definition}
According to the chosen viewpoint, we will alternatively say that $(Y^n)$ converges stably to the rule $\R$, to {the random kernel $\Gamma$}, or  to ${\cal I}_d$ in $(\bar \Omega,\bar \G, \R)$.

 Let us give two extensions, useful in the following:\\ 
\rmi  The boundedness requirement can by replaced by   the uniform integrability of the sequence $(H(\cdot, Y^n))$.\\
\rmii  The continuity of $H(\omega,\cdot)$ can be relaxed if there exists a subspace $A$ of $\bar{\Omega}$ with $\R^n(A) \rightarrow 1$ and $\R(A) = 1$, and such that $\P$-a.s., the restrictions of $H(\omega,\cdot)$ to $A_\omega$ are continuous, $\Gamma(\omega,\cdot)$-a.s.
%
%CCCCCCCCC
\subsubsection{Relative compactness}  We recall the link between tightness properties and relative compactness properties for the stable convergence:
\begin{lemma}[\cite{JacodMemin1981}]
\label{EquivalTightnessStableComp}
Let $(Y_n)$ be a sequence of random variables, with  distribution $(\mu_n)$. Then,\\
 $(\R^{Y_n})$ is  relatively compact in $\cR(\P,\cX)$ iff $(\mu_n)$ is tight in $\cX$.  
In particular, if $(\mu_n)$ converges weakly to $\mu$, then  $(Y_n)$ converges stably along a subsequence to a rule $\R$ such that $\mu = \R_{\cX}$. 
\end{lemma}
\noindent For simplicity, we also say ``$(Y_n)$ or $(\Gamma^{Y_n})$ is stably relatively compact  in $\cR(\P,\cX)$'' and use the abuse of language ``$(Y^n)$ is tight in $\cX$''. 
%===============================================
\subsection{Stable limits of  $\Ndemeps$}
%===============================================================================
\label{SectionCvgDemSystem}
 We now come back to the study of the the demographic counting processes $(\Ndemeps)$ for the two timescale BDS  \eqref{EqTwotimescaleBDS}. \\
 We denote by $\cal A^{2p}$  the closed subspace of the Skorohod space of c\`adl\`ag functions from $\mathbb{R}^+$ to $\N^{2p}$, composed of $\N^{2p}$-valued functions whose components only have unit jumps and no common jumps. We recall that for $[\alpha]\in \cal A^{2p}$,   $[\alpha]_t = \alpha(\cdot \wedge t)$. 
The canonical filtration is  denoted by $(\F^{\cal A}_t) =  (\sigma(\alpha(s) \; ; \; s\leq t, \; [\alpha] \in \cal A^{2p} ))$.\\

Stable limits of   $(\Ndemeps)_{\epsilon}$  are defined according to the second viewpoint (Definition \ref{StableCvgScdViewpoint}). The demographic processes $(\Ndemeps)$ are considered as $\cX = \cal A^{2p}$-valued random variables, and the enlarged space is
\begin{equation}
\label{enlargedspace}
(\bar \Omega, (\bar \G_t)) = (\Omega\times \cal A^{2p}, (\G_t \otimes \F^{\cal A}_t)), \quad \text{with } \bbolN^{\dem}(\omega,[\alpha]) = [\alpha].
\end{equation} 
Then,    $(\Ndemeps)$ converges stably to $\bbolN^{\dem}$  on   $(\bar \Omega, (\bar \G_t), \R^\dem)$ to  if 
\begin{equation*}
 \forall H \in \mathcal C_{bmc}(\Omega \times \cal A^{2p}),\quad \R^\dem[ H(\Ndemeps)] =\E[H(\Ndemeps)] \underset{\epsilon \to 0}{\longrightarrow} \R^\dem[ H(\bbolN^{\dem})].
\end{equation*} 
The strong domination  yields straightforward tightness properties. In addition, the property  is preserved  at the limit:
\begin{proposition}
\label{RelativeCompactnessNdem}
\rmi The family of  demographic counting processes $(\Ndemeps)$  is  stably relatively compact in $\cR(\P, \cal A^{2p})$. \\
\rmii Let $\R^\dem \in \cR(\P, \cal A^{2p})$ be a stable limit  of $(\Ndemeps)$.   Then,  $\bbolN^\dem$ is $\R^\dem$-a.s. strongly dominated by $G^\dem$: $\bbolN^\dem\prec G^\dem$.\\
\rmiii  If  $(\Ndemeps)$  converges stably to $\bar \bolN^\dem$ on $(\bar \Omega, \bar \G, \R^\dem)$,  then  the aggregated population processes $(Z^{\epsilon,\natural})$ converges  stably toward the process defined  by 
\begin{equation}
\label{LimitProcZnatural}
\bar X_t  = Z_0^\natural + \bar N^{b,\natural}_t - \bar N^{d,\natural}_t, \quad \forall t \geq 0.
\end{equation}
\end{proposition}
Note that  $\bar X$ is not the aggregation of an underlying heterogeneous population, since  the  one dimensional population is obtained when seeing the aggregated population $Z^{\epsilon, \natural}$ as a  function of $\Ndemeps$. 
\begin{proof}
\rmi  By Lemma \ref{EquivalTightnessStableComp}, it is sufficient to show that $(\Ndemeps)$ is tight in $\cal A^{2p}$.   Let $\epsilon >0$. 
Since $\Ndemeps \prec \bolG^\dem$,   $ G^\dem - \Ndemeps$ is a multivariate counting process.
Thus, 
for any stopping times $S \leq T $,
\begin{align*}
 0 \leq   G^\dem_S - \Ndemeps_S \leq  G^\dem_T - \Ndemeps_T,  
\text{ and thus } \;  0 \leq \Ndemeps_T - \Ndemeps_S  \leq  G^\dem_T -  G^\dem_S.
\end{align*}
This  yields  that  $(\Ndemeps)$ is tight by  the Aldous tightness criterion.\\
\rmii If $(\Ndemeps)$ converges stably along a subsequence to $\bbolN^{\dem}$ on $(\bar \Omega, (\bar \G_t), \R^\dem)$, then $(G^\dem -\Ndemeps)$ also  converges stably (along the same subsequence) to $ G^\dem -\bbolN^{\dem}$ on $(\bar \Omega, (\bar \G_t), \R^\dem)$. Thus, $G^\dem -\bbolN^{\dem}$ is a multivariate counting process, i.e. $\bbolN^{\dem} \prec  G^\dem$.\\
\rmiii $Z^{\epsilon,\natural}= F(\Ndemeps) = Z_0^\natural + N^{b,\epsilon,\natural} - N^{d,\epsilon,\natural} $ is a continuous function of $\Ndemeps$. Thus, if $(\Ndemeps)$ converges stably to $\bar \bolN^\dem$,  then $(Z^{\epsilon,\natural})$ converges stably to $F(\bar \bolN^\dem)$.
\end{proof}
%
%%%%%%%%%%%%%%%%
\subsection{Population variables $\tilde Z^\epsilon$ } 
\label{SectionPopVar}
 In order to establish additional properties  for stable limits of $(\Ndemeps)$, we may study the limit of their compensators $ A^\epsilon_t := \int_0^t \bolmu^\dem(s,Z^\epsilon_{s})\d s$.  The difficulty is that   $Z^\epsilon= Z_0 + \phi^s \odot \bolN^{s,\epsilon}  +\bolN^{b,\epsilon} - \bolN^{d,\epsilon}$  depends  on the  process $\bolN^{s,\epsilon}$ counting swap events. Due to the explosion  of swap events, the family of  population processes $(Z^\epsilon)_{\epsilon >0}$ is   not tight in $D(\mathbb R^+,\N^p)$. The idea   is  to  study a weaker version of the BDS population, by seeing $Z^\epsilon$ as an {\em $\N^p$-valued random variable} on $ \tilde \Omega := \Omega \times \mathbb R^+$, on a probability space well-suited to the study of the compensators.\\

 $\tilde \Omega$  is equipped with the optional $\sigma$-field $\mathcal O$ (generated by adapted c\`adl\`ag processes), and the probability measure $\tilde \P = \P \otimes \lambda^e$, with $\lambda^e(\d s) =e^{-s}\d s$. 
 For any  bounded variable $\tilde U(\omega,s )=U_s( \omega)$,
\begin{equation*}
\tilde \E (\tilde U)= \E [\int_{0}^{\infty} U_s \lambda^e(\d  s)]= \E[\int_{0}^\infty  U_s e^{-s} \d  s].
\end{equation*}
The family of BDS population is tight on this space:
\begin{proposition}[Population variables] 
\label{PopTightZ}
For $\epsilon>0$, let 
\begin{equation}
\tilde Z^\epsilon(\omega,s) =Z_s^\epsilon(\omega),\quad \forall \; (\omega,s) \in \tilde{\Omega}.
\end{equation}
Then, the  family of $\N^p$-valued random variables $(\tilde Z^\epsilon)$ defined on $(\tilde \Omega , \mathcal O,  \tilde \P)$  is stably relatively compact.
\end{proposition}
\begin{proof}
\noindent The stable relative compactness is once again a consequence of the construction by strong domination. Indeed, $\forall \epsilon >0$ $\tilde Z^{\epsilon,\natural}(\omega,s) \leq Z_0^\natural(\omega) + \tilde G^{b,\natural} (\omega,s) = Z_0^\natural(\omega) + G^{b,\natural}_s(\omega)$, 
which yields the tightness, and thus the relative compactness of the family of $\N^p$-random variables.
\end{proof}

%%%%%%%%%%%%%%%%%%%
\subsection{Joint stable limits of $\Ndemeps$ and $\tilde Z^\epsilon$} \label{SectionJointcvg}
 Informally,  the stable convergence of $\tilde Z^\epsilon$ on $\Omega \times \mathbb R^+$ corresponds to the  convergence of integrals of type $\E [\int_0^{\tau} \bollambda( s,Z^\epsilon_s) \d s]$.  However, we need a little bit more in order to carry out the study of the compensators: the joint stable convergence of $(\Ndemeps)$ and $(\tilde Z^\epsilon)$, this first requires to  define  $\Ndemeps$ and $\tilde Z^\epsilon$ on  the same probability space.
\paragraph{Extension of $(\Ndemeps)$  to $\tilde \Omega$ }This technical step is done  by defining on   $(\tilde \Omega, \mathcal{O})$:
\begin{equation}
\tilde{\bolN}^{\dem,\epsilon}(\omega,s)  =[\Ndemeps]_s(\omega) =\Ndemeps_{\cdot\wedge s}(\omega),\quad \forall (\omega,s) \in \tilde{\Omega}.
\end{equation}  
By Lemma \ref{LemmaNtilde}, the stable convergence of $(\tilde{\bolN}^{\dem,\epsilon})$ to a rule $\tilde \R^{\emph \dem}$  is equivalent to the stable convergence of $(\Ndemeps)$ to $\R^\dem$, with 
\begin{equation}
\label{EqLemmaA2}
\tilde{\R}^{\emph \dem} (\tilde h) = \R^{\emph \dem}[\int_0^\infty  h_s([\bar {\emph \bolN}^\dem]_s)\lambda^e(\d s)], \quad \forall \; \tilde{h} \in \cC_{bmc}(\tilde \Omega \times \cal A^{2p}).
\end{equation}
Stopping the trajectory at time $s$ ensures that $\tilde{\bolN}^{\dem,\epsilon}$ is $\cal O$-measurable. Thus, at the limit all processes considered are $(\bar \G_t  = \G_t \otimes \cal F^{\mathcal A}_t)$-adapted.
\paragraph{Joint stable limits}  We now turn to stable limits of  the family  of $\cal A^{2p}  \times \N^p$ variables  $((\tilde{\bolN}^{\dem,\epsilon},\tilde Z^\epsilon))$, defined on $(\tilde{\Omega},\cal O, \tilde P)$.\\
 Let $\tilde{\R}  \in \cR(\tilde \P, \cal A^{2p}\times \N^p)$  be  a stable limit of  $((\tilde{\bolN}^{\dem,\epsilon},\tilde Z^\epsilon))$.  $\tilde{\R}$ can be disintegrated with respect to its marginal $\tilde{\R}^\dem$  into\\[1mm]
 \centerline{$\tilde \R(\d \omega, \d s, \d [\alpha], \d z) =\tilde \R^{\dem}(\d \omega, \d s, \d [\alpha])\tilde \Gamma( \omega,  s, \d z)$,}\\[1mm]
where $\tilde{\Gamma}:= (\Gamma_s)_{s \geq 0}$ is a random probability kernel from $\tilde \Omega \times \cal A^{2p}$ to $\N^p$.\\
In particular, $\tilde \R^{\dem}$ is a stable limit of $(\tilde{\bolN}^{\dem,\epsilon})$. Combined with  \eqref{EqLemmaA2}, this yields the following decomposition:
\begin{lemma} 
\label{LemmaJointCvg}
Let $ \tilde{g}:= (g_s)_s \in C_{bm}(\tilde \Omega \times \cal A^{2p} \times \N^p)$. 
Stable limits  $\tilde{\R} \in \cR(\tilde \P, \cal A^{2p}\times \N^p) $ of $((\tilde{\bolN}^{\dem,\epsilon},\tilde Z^\epsilon))$ admit the following decomposition:
\begin{equation*}
 \tilde{\R}(\tilde{g}) = \R^\dem[\int_0^\infty \int_{\N^p} g_s([\bbolN^\dem]_s,z)  \Gamma_s( [\bbolN^\dem]_s,\d z)  \lambda^e(\d s)],
\end{equation*}
where  ${\R}^{\dem}$  is a stable limit of $({\bolN}^{\dem,\epsilon})$ and $(\Gamma_s)$ is the random  kernel defined above, which  can be interpreted as a stable limit of $(\tilde{Z}^\epsilon)$, ``conditionally to $ \tilde{\bolN}^{\dem,\epsilon}$''. \\
Such a couple $(\R^\dem,\Gamma)$ is called stable limit of $((\Ndemeps,\tilde{Z}^\epsilon))$.
\end{lemma}
\begin{proof}
Using the disintegration introduced above, 
$ \tilde{\R}(\tilde{g}) =  \tilde \R^\dem\big(\int_{\N^p} \tilde g(\cdot ,z) \tilde \Gamma( \cdot,\d z)\big)$.\\
Applying Lemma \ref{LemmaNtilde} with $\tilde{h}([\alpha]) = \int_{\N^p} \tilde g(\cdot , [\alpha] ,z)  \tilde \Gamma(\cdot , [\alpha],\d z)$ yields the result. 
\end{proof}
By taking $\tilde{g}(\omega,s,[\alpha],z) = h(\omega,[\alpha]_{u})\bolmu^\dem(\omega,s,z)\indi_{[u\wedge \tau, t\wedge \tau]}(\omega,s)e^s$, we recover:
\begin{equation*}
\tilde{\E}[\tilde{g}([\tilde \bolN^{\dem,\epsilon}], \tilde{Z}^\epsilon)] = \E [h([\bolN^{\dem,\epsilon}]_{u})  \int_{u\wedge \tau}^{t\wedge \tau} \bolmu^\dem(s,Z^\epsilon_s)\d s] =  \E [h([\bolN^{\dem,\epsilon}]_{u})\big(A^\epsilon_{t \wedge \tau} - A^{\epsilon}_{u\wedge \tau}\big)].
\end{equation*}
Thus,  if $((\Ndemeps, \tilde{Z}^\epsilon))$ converges stably to $(\R^\dem, \Gamma)$,  applying  informally the convergence to $\tilde g$ yields:
\begin{align*}
\label{Eqjointcvg}
\E [h([\bolN^{\dem,\epsilon}]_{u})  \int_{u\wedge \tau}^{t\wedge \tau} \bolmu^\dem(s,Z^\epsilon_s)\d s] \to 
 \R^\dem [h([\bar \bolN^\dem]_{u}) \int_{u\wedge \tau}^{t\wedge \tau}\int_{\N^p} \bolmu^\dem(s,z)  \Gamma_s([\bar \bolN^\dem]_s,\d z)\d s] .
\end{align*}
\begin{remark}
\label{RemarqueKurtz}
 It is usual when studying limits of counting processes to study the limit {\em in distribution} of the compensators.  For instance,  in the spirit of  \cite{Kurtz1992}, we would have  studied limits of the random occupation measure $\Lambda^\epsilon (\d z,\d s) =\delta_{Z^\epsilon_s}(\d z) \d s$, since   $  A^\epsilon_t = \int_0^t \bolmu^\dem( s,Z^\epsilon_s) \d s =  \Lambda^\epsilon(\indi_{[0,t]}\bolmu^\dem)$. However,  the convergence in distribution of $(\Lambda_\epsilon)_\epsilon$ does not yields the convergence in distribution of $(A_\epsilon)_\epsilon$,  since $ \bolmu^\dem$ is a general random functional.
 
 The stable convergence applied on the space $(\tilde{\Omega},\mathcal{O}, \tilde{\P})$ rather  yields the convergence of ``conditional expectations''  of the compensators. Note also that here, it is straightforward that compensators' limits are absolutely  continuous with respect to the Lebesgue measure, which is not trivial   when considering the convergence  in distribution of the occupation measures (see Lemma 1.4 in \cite{Kurtz1992}). 
\end{remark}
\subsection{Averaging results}
We close this section with a general averaging result,  showing that  at the limit, the  intensity of the demographic processes $(\Ndemeps)$ is the demographic intensity functional $\bolmu^\dem$, averaged against  a stable limit $\Gamma$  of $(\tilde Z^\epsilon)$. In particular,  the  aggregated averaged birth and death intensity only depend on the past demographic events and the environment. 
 We also obtain specific  support properties for the kernel $\Gamma$. 
These results can be applied to obtain weak convergence results for the demographic counting processes, as detailed in Section \ref{SectionWeakConvergence}.
\begin{theorem}
\label{ThIdentificationResult}
Let $(\R^{\emph \dem},{\Gamma})$ be 
a stable limit of $({\emph{\bolN}}^{\dem,\epsilon}, \tilde Z^\epsilon)$.\\
\rmi The stable limit $\bar{\emph \bolN}^\dem$ of $(\Ndemeps)$ defined on  $(\bar \Omega , (\bar \G_t) , \R^{\emph\dem})$  is strongly dominated by $ G^\dem$, and  has the $(\bar \G_t)$- averaged intensity
\begin{equation}
 {\Gamma}_t([\bar{\emph{\bolN}}^\dem]_t,\bolmu^\dem) = \int_{\mathbb N^{p}} \bolmu^\dem(t,z) \Gamma_t([\bar{\emph{\bolN}}^\dem]_t,\d z).
\end{equation}
\rmii  As a consequence, the stable limit   $\bar X = Z_0^\natural + \bar N^{b,\natural}- \bar N^{d,\natural}$  of  $(Z^{\epsilon, \natural})$ has  the $(\bar \G_t)$  birth  (resp. death) intensity:
\begin{equation}
{\Gamma}_t([\bar{\emph{\bolN}}^\dem]_{t^-},\mu^{b,\natural})= \sum_{i=1}^p {\Gamma}_t([\bar{\emph{\bolN}}^\dem]_{t^-},\mu^{b,i}) ,
\end{equation}
(resp. ${\Gamma}_t([\bar{\emph{\bolN}}^\dem]_{t^-},\mu^{d,\natural})= \sum_{i=1}^p{\Gamma}_t([\bar{\emph{\bolN}}^\dem]_{t^-},\mu^{d,i})$).
\end{theorem}
Uniform integrability assumptions on the compensators are also not necessary here. The latter assumption is replaced by the fact that the sequence  $(S_p)$ of jump times of $\bolG^\dem$ is a localizing sequence for all local martingales $(\Ndemeps - A^\epsilon)$ (Lemma \ref{CorFamilyStrongDomination}). 
 
Another advantage of this approach  is that all   results can be extended directly to the case when  event intensities are path-dependent functionals of the demographic counting processes (and thus of the aggregated population).
\begin{proof}[\bol{Proof of Theorem \ref{ThIdentificationResult}}]
Let $(\R^\dem,\Gamma)$ be a stable limit of $({\emph{\bolN}}^{\dem,\epsilon}, \tilde Z^\epsilon)$. Then,  there exists a subsequence $\epsilon_k \rightarrow 0 $ along which $(( \bolN^{\dem,\epsilon_k}, \tilde Z^{\epsilon_k}))$ converges stably  to $(\R^\dem,\Gamma)$.\\
Let $ 0\leq u<  t$ and  $h_u\in C_{bmc}(\Omega\times \cal A^{2p})$ a $\G_{u}\otimes \F^{\cal A}_{u}$-measurable function ($h_u([\alpha]) = h_u([\alpha]_u)$).  By Lemma \ref{CorFamilyStrongDomination}, for all $p\geq 0$ $(\bolN^{\dem,\epsilon_k}_{\cdot \wedge S_p} - A^{\epsilon_k}_{\cdot \wedge S_p})$ is a martingale, so that:
\begin{equation}
\label{ProofMartingaleProp}  \E[h_u(\bolN^{\dem,\epsilon_k}) (\bolN^{\dem,\epsilon_k}_{t\wedge S_p} -  \bolN^{\dem,\epsilon_k}_{u\wedge S_p})]  = \E [h_u(\bolN^{\dem,\epsilon_k}) \int_{u\wedge S_p}^{t\wedge S_p} \bolmu^\dem(s,Z^{\epsilon_k}_s)\d s ].  
\end{equation} 

\bol{Step 1} Let $H([\alpha]) = h_u([\alpha])(\alpha(t\wedge S_p) -\alpha(u\wedge S_p))$. The left hand side of \eqref{ProofMartingaleProp} is equal to $\E[H(\Ndemeps)]$.
$H(\omega,\cdot)$ is not  bounded.  However, for all $\epsilon>0$, $\vert H( \Ndemeps) \vert \leq \Vert h_u\Vert_{\infty}(\bolG^\dem_{t\wedge S_p}-\bolG^\dem_{u \wedge S_p})$ which is in $\L^1(\Omega,\G,\P)$, and thus $(H( \Ndemeps))$ is  uniformly integrable.\\
By Lemma \ref{LemmaTechContinuity}, the continuity assumption can be replaced by the fact that $H(\omega,\cdot)$ is continuous on the  space of  paths $[\alpha] \prec G^\dem(\omega)$: 
$$A= \{(\omega,[\alpha]) \in \Omega \times \cal A^{2p} ; \; [\alpha] \prec  G^\dem(\omega) \text{ and }\Delta \alpha(t) = \Delta \alpha (u) = 0\}.$$
 Since $\Ndemeps \prec \bolG^\dem$ and $\bar \bolN^\dem \prec \bolG^\dem$ (Proposition \ref{RelativeCompactnessNdem}), $\R^\epsilon(A)=\E[\indi_A(\Ndemeps)] =\R^\dem(A)= 1$. 
Thus,  the stable convergence of $(\bolN^{\dem,\epsilon_k})$ to $\R^\dem$ can be applied to $H$. \\

\bol{Step 2} Now, let $\tilde g(s,[\alpha],z) = \indi_{]u \wedge S_p,t\wedge S_p]}(s) e^s h_u([\alpha]_u) \bolmu^\dem(s,z)$. The right hand side of 
\eqref{ProofMartingaleProp} is equal to $\tilde \E [\tilde g(\tilde \bolN^{\dem, \epsilon_k},\tilde 
Z^{\epsilon_k} )]$. \\
 Let $\beta^\dem(s, G^{b,\natural})$ be the intensity process of $\bolG^\dem$,  defined in Proposition \ref{Birth event process}.
Since $\Ndemeps \prec G^\dem$  and by Assumption \ref{HypCoxBirthDom}, $ \bolmu^\dem(s,Z^{\epsilon_k}_s) \leq  \beta^\dem(s, G_s^{b,\natural})$, and thus 
 $$ \tilde \E[\vert  \tilde{g}([\tilde{N}^{\dem,\epsilon_k}],\tilde{Z}^{\epsilon_k})\vert ]  \leq  \Vert h_u\Vert_{\infty} \E[\int_{u\wedge S_p}^{t \wedge S_p} \beta^\dem(s,G_s^{b,\natural})\d s] \leq  p\Vert h_u\Vert_{\infty}.$$
 
\bol{Step 3} Hence, the sequence $(\tilde g(\tilde \bolN^{\dem, \epsilon_k},\tilde 
Z^{\epsilon_k} ))_\epsilon$ is uniformly integrable,  and  the joint stable convergence  (Lemma \ref{LemmaJointCvg}) can be applied to $\tilde{g}$.  By taking the stable limits in both sides of Equation \eqref{ProofMartingaleProp}, we obtain that:
\begin{equation*}
\R^\dem [h_u(\bar \bolN^\dem) (\bar \bolN^\dem_{t\wedge S_p} -\bar \bolN^\dem_{u\wedge S_p})]  = \R ^\dem[h_u(\bar \bolN^\dem)  \int_{u\wedge S_p}^{t\wedge S_p} \int_{\N^{p}}\bolmu^\dem(s,z){\Gamma}(s,[\bar \bolN^\dem]_s,\d z)\d s ],
\end{equation*} 
for all $\G_u\otimes \cal F^{A}_u$-measurable $h_u$, which achieves the proof.
\end{proof}
\subsubsection{Properties of  the kernels $\Gamma$} First,  the relation $Z^{\epsilon,\natural} = \sum_{i=1}^p Z^i = Z_0^\natural +  N^{b,\natural,\epsilon} - N^{d,\natural,\epsilon}$ between the aggregated population and  $\Ndemeps$, is translated into a support property for $\Gamma$. The space of populations of size $n$ is denoted by $\U_n =\{ z \in \N^p \; ; \sum_{i=1}^p z^i = n\}$. 
\begin{corollary}
\label{CorGammaSupport}
Let $(\R^\dem, \Gamma)$ be a stable limit of  $( \bolN^{\dem, \epsilon}, \tilde Z^\epsilon)$, and $\bar X = Z_0^\natural + \bar N^{b,\natural} - \bar N^{d,\natural}$ the limit aggregated process. The support of $(\Gamma_s)$ is included in $\U_{\bar X_s} $, i.e.
\begin{equation}
\label{supportGamma}
\Gamma_s([\bar{\emph{\bolN}}^\dem]_s,\U_{\bar X_s})=1 \quad \R^{\emph{\dem}} \otimes \d s \; a.s.,
\end{equation}
\end{corollary}
\begin{proof}
Let $F= \{(\omega,s,\alpha,z); \;  z^\natural = Z_0^\natural(\omega)+ \alpha^{b,\natural}(s) -\alpha^{d,\natural}(s)\}$,  We have $\forall \;  \epsilon >0$:
\begin{equation*}
\tilde \E[\indi_F(\cdot,\tilde \bolN^{\dem,\epsilon}, \tilde{Z}^\epsilon)]=\E[\int_0^\infty \indi_{\{Z^{\epsilon,\natural}_s = Z_0^\natural +  N^{b,\natural,\epsilon}_s - N^{d,\natural,\epsilon}_s\}}\lambda^e(\d s)]= 1.
\end{equation*}
By the extension of the portemanteau theorem to the stable convergence (\cite{JacodMemin1981}, \cite{castaing2004young}), 
we obtain that\\
$\R^\dem [\int_0^\infty\Gamma(\cdot,s,[\bar \bolN^\dem]_s,\U_{\bar Z^\natural_s})\lambda^e(\d s) ]=1$.
\end{proof}
\noindent The second property expresses that limit kernels must cancel the effect of fast  swap events.  
The martingale problem verified by the two timescale BDS population supports the intuition. 

For a BDS intensity functional $\bolmu =(\mu_\gamma)_{\gamma \in \J}$, a random operator  $(L_t^\sw)_{t\geq 0}$  (resp. $ (L^\dem_t)_{t\geq 0}$) can be associated with swap events (resp. demographic events),
\begin{equation}
L^\sw_t f(z) = \sum_{i,j =1 \atop i\neq j}^p\big( f(z+\phi(i,j)) - f(z)\big)\mu^{(i,j)}(t,z), \quad \forall t \geq 0, \; z \in \N^p,
\end{equation} 
(reps. $L^\dem_t f(z) = \sum_{j=1}^p \big( f(z+\bolej) - f(z)\big)\mu^{(b,i)}( t,z) +  \sum_{i=1}^p \big( f(z-\bolei) - f(z)\big)\mu^{(d,i)}( t,z)$).\\

\noindent For any bounded function $f$ on $\N^p$, 
\begin{equation*}
f(Z^\epsilon_t) - f(Z_0) - \int_0^t L^\dem_s f(Z_s^\epsilon) \d s  - \frac{1}{\epsilon} \int_0^t  L^\sw_s f(Z_s^\epsilon) \d s \text{ is a $(\G_t)$-local martingale.}
\end{equation*}
Thus, the  swap  part explodes 
when $\epsilon \to 0$.   At the limit   $(L_s^\sw f)$ is averaged against a stable limit $(\Gamma_s)$ of $\tilde{Z}^\epsilon$, and  intuitively, this means that $\Gamma$ must satisfy $\Gamma_s ([\bar N^\dem]_s, L_s^\sw f) =0$. An interpretation of this property  in terms of invariant measures for pure Swap processes is detailed in the next Section. 
\begin{proposition}
\label{ThGammaInvariant}
Let $(\R^\dem, \Gamma)$ be a stable limit of  $( \bolN^{\dem, \epsilon}, \tilde Z^\epsilon)$, and $\bar X= Z_0^\natural + \bar N^{b,\natural}- \bar N^{d,\natural}$ the limit aggregated process. Then, for all bounded functions $f$ on $\mathbb N^p$, 
\begin{equation}
\label{AveragingGamma}
\Gamma_s ([\bar N^\dem]_s, L_s^\sw f) = \sum_{z \in \U_{\bar X_s}} L_s^\sw f(z)\Gamma_s([\bar N^\dem]_s, \d z) = 0, \quad \R^\dem \otimes \d s \; a.s.
\end{equation}
\end{proposition}
\begin{proof}
Let $f \in \cC_b(\N^p)$, $0 \leq u<t$ and $h_u \in \cC_{bmc}(\Omega \times \cal A^{2p})$ a $\G_u \otimes \F^{\cal A}_u$-measurable functional.
\begin{align*}
 & \E[h_u(\Ndemeps)  \left(f(Z^\epsilon_{t\wedge S_p}) - f(Z^\epsilon_{u\wedge S_p})\right)]  =\\  
 & \E[h_u(\Ndemeps) \int_{u\wedge S_p}^{t\wedge S_p}L_s^\dem f(Z^\epsilon_s)\d s + h_u(\Ndemeps)\frac{1}{\epsilon} \int_{u\wedge S_p}^{t\wedge S_p}L_s^\sw f(Z^\epsilon_s)\d s].
\end{align*}
By multiplying the equation  by $\epsilon$ and letting $\epsilon\rightarrow 0$ along the appropriate subsequence  we obtain by
using the same uniform integrablity argument as in the proof of Theorem \ref{ThIdentificationResult}:
$$\R^\dem [h_u(\bar \bolN^\dem)\int_{u\wedge S_p}^{t\wedge S_p} \Gamma_s(  [\bar \bolN^\dem]_s, L_s^\sw f ) \d s ]=0.$$
The foregoing yields that $\int_0^t \Gamma_s([ \bar \bolN^\dem]_s, L_s^\sw f ) \d s $ is a $(\bar \G_t)$-local martingale with finite variations, and is thus the null process, so that: 
\begin{equation*}
\Gamma_s( [\bar \bolN^\dem]_s, L_s^\sw f)    =0, \quad \forall f\in \cal C_b(\N^p) \quad \R^\dem \otimes \d s \text{ a.s}.
\end{equation*}
\end{proof}
\section{Convergence in distribution of the demographic processes in random environment}${}$\\
\label{SectionWeakConvergence}

\noindent In Section \ref{SectionCvgDemSystem}, we  showed that stable limits $\bar \bolN^\dem$ 
of the demographic counting processes $(\Ndemeps)$ are strongly dominated by $\bolG^\dem$ and have an  intensity of form $\Gamma_t([\bar \bolN^\dem]_t, \bolmu^\dem)$, where the averaging kernel $(\Gamma_s)$ verifies the support constraints \eqref{supportGamma} and the averaging constraint \eqref{AveragingGamma}. 

These results can be applied to identify limit distributions of $(\Ndemeps)$, since all limits for the convergence in distribution can be realized as stable limits $\bar \bolN^\dem$,  and thus inherit the  properties cited above. 
To conclude this paper, we give such  an  application, under an additional assumption yielding  the uniqueness of  the averaging kernels $\Gamma$. As a consequence,  the aggregated ``macro'' population processes converge in distribution toward a birth-death process in random environment. We close this section with  a toy model, showing the emergence of nonlinearity  in  mortality rates in the presence of fast swap events.
%
%%%%%%%%%%%%%%%%
 \subsection{Main result}
We first come back to the interpretation of Proposition  \ref{ThGammaInvariant}  in terms of invariant measures for pure Swap processes.
 \subsubsection{Markov Pure Swap process with ``frozen environment''}  Pure Swap processes are population processes in which  only swap events occur.   For instance, when  swap intensity functionals $\mu^{(i,j)}$ are deterministic functions $\mu^{i,j} : z \in \N^p \rightarrow \mu^{(i,j)}(z) $, a pure Swap process $S$ of intensity functions $(\mu^{(i,j)})$ is a Continuous Time Markov Chain (CMTC), of generator $\mathcal L = L^\sw$. Since swap events don't change the size of the population, if   $S_0 \in \U_n$, then  $S$ remains in the space $\U_n$ of populations of size $n$ at all time. In this  particular case,  Proposition  \ref{ThGammaInvariant}  means that   $\Gamma_s([\bar{\emph{\bolN}}^\dem]_s,\cdot)$, whose support is $\U_{\bar X_s}$, is an {\em invariant measure} of the Markov Swap $S$ starting with an initial population of size $\bar X_s = Z_0^\natural + N^{b,\natural}_s - N^{d,\natural}_s$. 
 
 In the general case,  $(L^\sw_s)$ is a random time-dependent operator. However, for a fixed $(\omega,s)$, $\cal L = L^\sw_s(\omega)$  can be seen as the generator  of a  fictitious Markov pure Swap, where the random environment has been frozen at time $s$,  with intensity functions  $g^{(i,j)}(z) = \mu^{(i,j)}(s,z)$. Thus,  Proposition \ref{ThGammaInvariant} can be interpreted as  $\Gamma_s([\bar{\emph{\bolN}}^\dem]_s,\cdot)$ being an invariant measure of the  Swap CTMC of generator $\mathcal{L}$ with frozen environment,  starting with an initial population of size $\bar X_s$, this being true a.s. for all  fixed $(\omega,s)$.
 \subsubsection{Main assumption and convergence result} 
We now assume that all states of the finite space $\U_n$ are attainable by the Swap  CTMC of generator $\cal L = L_s^\sw(\omega)$, so  that each fictitious Markov  Swap process  has  a unique  invariant measure on each finite space $\U_n$. 
This    assumption  is essentially a structure property for the model, depending only on the swap intensity functional $\bolmu^s$.
\begin{hyp}
\label{HypInvariantMeasure}
For each population size $n$,
there exists a unique 
predictable probability kernel $\pi_s(\omega,n,\d z)$ on $(\Omega, \G, \P)$ with support in $\U_n$,  such that  for all $f: \U_n \mapsto \mathbb R$, 
\begin{equation}
\label{condconvergence}
\pi_s( n, L_s^\sw f) =0  ,  \quad  \P \otimes \d s \;  a.s. 
\end{equation}
\end{hyp}
Assumption \ref{HypInvariantMeasure}  yields that    $\Gamma_s([\bar \bolN^\dem]_s,\cdot) =\pi_s(\bar X^\natural_s, \cdot)$,  for any joint stable limit $(\R^\dem,\Gamma)$. In particular,  under  Assumption \ref{HypInvariantMeasure}   limit kernels $\Gamma_s$ only depend on $[\bar \bolN^\dem]_s$ through  the aggregated population $\bar X_s$.  This assumption is  a sufficient assumption for  the demographic processes to converge in distribution. 

\begin{theorem}[Convergence in distribution  of the demographic processes]
\label{weak convergence}${}$\\
\rmi Under Assumption \ref{HypInvariantMeasure}, $(\emph{\bolN}^{\dem,\epsilon})$ converges in distribution in $\cal A^{2p}$. 
The limit distribution can be realized on the initial probability space $(\Omega,(\G_t),\P)$, as the distribution of  the $2p$-multivariate counting process   $\cal N^\dem = (\cal N^b, \cal N^d)$,  solution of:
\begin{equation}
\label{EqNdemLimit}
\d  \cal N^\dem_t = \bolQ^\dem (\d t, ]0,\pi_t ( X_{t^-},\bolmu^\dem) ]), \quad X_t = Z_0^\natural +  \cal N^{b,\natural}_t -\cal N^{d,\natural}_t.
\end{equation} 
\rmii The aggregated population processes $(Z^{\epsilon,\natural})$ converge in distribution to the   Birth-Death process  in random environment $X$, with random $(\G_t)$-birth and death intensities functionals defined by 
\begin{equation}
\pi_t(n, \mu^{b,\natural}) = \int_{\U_{n}}\mu^{b,\natural}(t,z){\pi}_t(n,\d z) \text{ and} \quad \pi_t(n, \mu^{d,\natural}) = \int_{\U_n}\mu^{d,\natural}(t,z){\pi}_t(n,\d z), \forall n \geq 0.
\end{equation}
\end{theorem}
\begin{remark}
\rmi The aggregated population converges toward an ``autonomous'' Birth-Death process in random environment.  In particular,  swap events generates a non-linear dependence of aggregated birth and death intensities on the population size, which are averaged against  invariant measures $\pi_t(n,\cdot)$. These averaging measures also depend non-trivially on the random environment.\\
\rmii In practice, it might be difficult to keep track of swap events,  and thus to have a precise estimate of the number of individual in each subgroup. The approximation of the demographic counting process  by $\mathcal{N}^\dem$  is useful since  it allows demographic rates to be approximated by functional depending only on the global size of the population.
\end{remark}
The proof of Theorem \ref{weak convergence}  relies on two ingredients. The first one is Corollary \ref{CorGammaSupport} and  Theorem  \ref{ThGammaInvariant}, which allow us to show that under Assumption \ref{HypInvariantMeasure}, all stable limits for the demographic processes have the same intensity functional.   
The second ingredient is the Corollary \ref{UniquenessDomination},  which  allows us to conclude since all stable limits have the same intensity functional {\em and} are are strongly dominated  by $\bolG^\dem$ (Proposition \ref{RelativeCompactnessNdem}), and thus have the same distribution, by Corollary \ref{UniquenessDomination}.
\begin{proof}[Proof of Theorem \ref{weak convergence}] 
 Let's prove that all stable limits $\R^\dem$ of $(\Ndemeps)$ have the same marginal on  $\mathcal A^{2p}$, or in other words, that  $\bar \bolN^\dem = (\bar N^b , \bar N^d)$  has the same distribution under all  limit rules $\R^\dem$. 
Let $\R^\dem$  be such a rule, and $\bar X = Z_0^\natural + \bar N^{b,\natural} - \bar N^{d,\natural}$ the limit aggregated process.\\
By Theorem \ref{ThIdentificationResult}, there exits a random kernel $\Gamma$ such that $\bar \bolN^\dem$ has the $(\bar \G_t)$ intensity $\Gamma_t([\bar \bolN^\dem]_t, \bolmu^\dem)$. By Corollary \ref{CorGammaSupport}  and Proposition  \ref{ThGammaInvariant}, there exists $B \in \mathcal{O} \otimes \mathcal{F}^{\cal A}_\infty$ with 
$\tilde \R^\dem(B) =1$, such that
$$\forall \; (\omega,s, [\alpha]) \in B, \quad \Gamma_s([\bar \bolN^\dem]_s,\U_{\bar X_s}) =1 \quad \text{and} \quad  \Gamma_s( [\bar N^\dem]_s, L_s^\sw f) = 0, \quad \forall \; f  \in \cC_{b}(\N^p).$$
By Assumption  \ref{condconvergence},  there also exists $ A \in \mathcal O$ with $\tilde P(A) = \tilde \R^\dem(A) =1$, such that $\forall \; (\omega,s) \in A$ and  $n \geq 0$, $\pi_s$ is the unique  measure verifying 
$$ \pi_s(n,\U_{n}) =1, \quad \text{and} \quad  \pi_s(n, L_s^\sw f) = 0, \quad \forall \; f  \in \cC_{b}(\N^p).$$
Thus, by setting $n = \bar X_s([\alpha])$,
$$\Gamma_s = \pi_s(\bar X_s , \cdot ) \quad \forall  \;  (\omega,s, [\alpha]) \in B\cap A.$$
Hence,  $\bar \bolN^\dem$ have the $\bar \G_t$ intensity functional $(\pi_t(\cdot,  \bolmu^\dem))$ and is strongly dominated by $\bolG^\dem$ (Proposition \ref{RelativeCompactnessNdem}, (ii)). 
By a direct application of Corollary  \ref{UniquenessDomination},  this means that  $\bar \bolN^\dem $ has the same distribution under all  limit rules $\R^\dem$ and thus $(\Ndemeps)$  converges in distribution. \\
To conclude the proof,  it remains to show that the solution  $\mathcal N^\dem$ of
$$\d  \mathcal N^\dem_t = \bolQ(\d t, ]0,\pi_t( X_{t^-},\bolmu^\dem) ]), \quad  X_{t} = Z_0^\natural + \mathcal N^{b,\natural}_t - \mathcal N^{d,\natural}_t$$
is well-defined and strongly dominated by $G^\dem$. 
The intensity  $ (k_t g_j(G^{b,\natural}))_{j=1\dots p}$ of $\bolG^\dem$ only depends on $G^\dem$ through $G^{b,\natural}$, and by the  domination Assumption \ref{HypCoxBirthDom}, $\bolmu^\dem(t, z) \leq (k_t g_j(n))_{j=1\dots p}$, for all $z \in \U_m$ with $m \leq n$.
As a consequence, for all $m \leq n$,
$$ \pi_t( m , \bolmu^\dem) = \int_{\U_{m}} \bolmu^\dem(t,z)\pi_t(m ,\d z)  \leq (k_t g_j(n))_{j=1\dots p}.$$ 
Thus, the intensity functional of $\mathcal{N}$ is strongly dominated by that of $G^\dem$, and by application of  Proposition \ref{comparison point process},  $\mathcal{N}^\dem$ is well-defined and strongly dominated by $\bolG^\dem$.
\end{proof}
\subsection{ A toy model}
We now give a toy example,   inspired from \cite{AUGER1995}, modeling the impact of spatial heterogeneity on a population living in a fragmented habitat.
\subsubsection{Model} The population's habitat is assumed to be divided on several patches, either favorable (type $1$) or unfavorable (type $2$).  The  population is  thus  structured   in  two subgroups  describing the number of individuals in each type of  habitat. The evolution of individuals is influenced by the type of patch in which they are located, as well as the random environment (weather, availability of resources, human action...),  which is also  correlated to the evolution of the population.

 Each individual in subgroup $i$ can die at  a stochastic individual rate $d^i_t$. The favorable subgroup 1 has a lower death rate: $d^1_t \leq d^2_t$ a.s.  Birth intensity functionals are also assumed to be linear (each individual gives birth at rate $b_t$ which is the same in both subgroup), with immigration at the stochastic rate $\lambda_t$.
 \begin{equation}
 \mu^{(d,i)}(t,z) = d^i_t z^i , \quad \mu^{(b,i)}(t,z) = b_t  z^i + \lambda_t,\quad i=1,2
 \end{equation}
When the population size increases individuals swap more often to the unfavorable habitat (type 2), for instance due to lack of ressources,  at a rate proportional to the population size $z^\natural =z^1 + z^2$, while, individuals swap from subgroup 2 to subgroup 1 at a ``constant'' stochastic rate, so that:
\begin{align}
&\mu^{(1,2)}(t,z) = (k_t^{12}z^\natural) z^1, \quad \mu^{(2,1)}(t, z) = k_t^{21}z^2, \\
 \nonumber  L_t^\sw & f(z) = (k_t^{12}z^\natural) z^1(f(z+ \bol e_2 -\bol e_1) - f(z)) + k_t^{21}z^2(f(z+ \bol e_1 - \bol e_2) - f(z)).
\end{align}
When the habitat fragmentation is important, individuals  migrates frequently between the different patches, which impact the population viability (\cite{PICHANCOURT2006}).  This situation is described by the two timescale BDS:
\begin{align}
& Z^{i, \epsilon}_t = N^{( b,i), \epsilon}_t - N^{(d,i),\epsilon}_t +  N^{(j,i),\epsilon}_t - N^{(i,j), \epsilon}_t , \quad i=1,2 \; j\neq i.\\
\nonumber & \d N^{(d,i),\epsilon}_t = Q^{(d,i)}(\d t, ]0, d_t^i Z^{i,\epsilon}_{t^-}] ), \quad  \d N^{(b,i),\epsilon}_t = Q^{(b,i)}(\d t, ]0, b_t Z^{i,\epsilon}_{t^-}+ \lambda_t] ), \\
&  \nonumber \d N^{(i,j),\epsilon} = Q^{(i,j)}(\d t, ]0, \frac{1}{\epsilon} \mu^{(i,j)}(t,Z^\epsilon_{t^-})]).
\end{align}
\subsubsection{Averaging kernel} 
The pure  swap CTMC  of generator $\cal L = L^\sw_s$  (the environment is frozen  at  $s$)  can be reinterpreted as follow: all individuals in the swap population evolve as ``independent CTMC'' on the state space $\{1,2\} = \{(1,0),(0,1)\}$, with  transition rates depending on the total number of individuals $n$.  The invariant measure for  such an  individual is:
\begin{equation*}
p_s^1(n)=  \dfrac{1}{ \alpha_s n + 1} \quad  p_s^2(n)=  \dfrac{\alpha_sn}{ \alpha_s n + 1}, \quad \text{with } \alpha_s = \frac{k_s^{12}}{k_s^{21}}.
 \end{equation*}
 The stationary measure $\pi_s(n,\cdot)$ on $\U_n$ of the Swap CTMC with frozen environment  is characterized by the distribution $\pi_s^1(n,\cdot)$ of the number of individuals in  subgroup 1 (since $z^2 = n - z^1$). By the foregoing,  $\pi_s^1(n,\cdot)$ is a binomial distribution with parameter $(n, p^1_s(n))$.
\subsubsection{Convergence of the demographic counting  process}   By Theorem \ref{weak convergence} the process $\bolN^{\dem,\epsilon}$  counting birth and death events in each subpopulation converges to the process $\mathcal{N}^\dem = (\mathcal{N}^b, \mathcal{N}^d)$ solution of 
\begin{align*}
& \d \mathcal{N}^{(d,i)}_t  = Q^{(d,i)} (\d t , ]0,  d_t^i X_{t^-}p_t^i({X}_{t^-})]),\quad \d \mathcal{N}^{(b,i)}_t  = Q^{b,i} (\d t , ]0,  (b_t + \lambda_t) {X}_{t^-}p_t^i({X}_{t^-})]),
\end{align*}
where $({X}_{t} = Z_0^\natural + \sum_{i=1}^2 \mathcal{N}^{(b,i)}_t - \mathcal{N}^{(d,i)}_t)$ is a Birth-Death process  in random environment which approximate the dynamic of the aggregated population. In particular,  its death intensity functional  is :
\begin{align}
&  \pi_t (n,\mu^{d,\natural} )   =    \big(d^1_t  p_t^1(n) + d^2_t p_t^2(n) \big) n=  \frac{d_t^1}{1 +  \alpha_t n }\big( 1  + \alpha_t w_t n)n, \quad \text{with } w_t =\frac{d_t^2}{d_t^1}.
 \end{align}
\subsection{Comments and perspectives}

 The previous example shows that nonlinearities can emerge due to frequent swap or migration events.  The aggregated mortality rates depend  on  the population size: when the population is small, the aggregated mortality rate is close to the death rate of the favorable subgroup 1. If the population size increase, for instance if the entry rate $\lambda$ increases, the aggregate mortality will  increases, resulting in  a decrease of  the number of individuals in the population and thus of  the aggregated  mortality rate. Thus, swap events can induce a regulation of the population size, whose equilibrium  also depends  on the random  environment.

 More generally, aggregation  results of this paper give a way to generate new birth-death models, with averaged  stochastic intensities,  reflecting the heterogeneity of the underlying population and taking into account the environment. This is particularly relevant for human populations  for which all  standard mortality rates models are  linear, which seems quite unrealistic in light of the observed complexity in human populations. However, in order to be well-suited to human populations,  the results  should be extended to the framework of age-structured populations in further work.

Sometimes, the only observed processes are the  aggregated population or  the demographic counting process $\bolN^\dem$.  Then, the problem can be seen as inverse problem.  If the demographic process is modeled by an multivariate counting process such as in \eqref{EqNdemLimit},  with intensities which  are a  mixture of subgroup-specific intensities, the converse problem would be to find a BDS population ``compatible" with the given aggregated structured, for instance by minimizing a given distance between the two populations. 

\appendix

\section*{Appendix}
\section{Proof of Proposition \ref{PropReciproqueComparaison} and Corollary \ref{UniquenessDomination} }
\label{AppendixProof}
\begin{proof}[Proof of proposition \ref{PropReciproqueComparaison}]
\rmi By construction,   $\indi_{\{(s,u) ; \;   u \leq \phi_s \}} \hat Q'(\d s,\d u )= \hat Q^N(\d s,\d u)$, whose marginal in time is $N$.  \\
\rmii  $Q^N$  has the $(\G_t \vee \F^{U,V}_t)$-intensity measure $\lambda_t \d t \otimes \d u$. Thus,  for any  non-negative predictable functional $H(\omega,t, u)$,   
\begin{align*}
 \E[\widehat Q^N(H) ]&=   \E[\int_0^\infty \int_0^{1}H(s, u\phi_s )Q^N(\d s, \d u)] = \E[\int_0^\infty\int_0^{1} H(s, u\phi_s )\lambda_s\d s \,\d u ] \\
&= \E[\int_0^\infty \int_0^{\phi_s} H(s,v) \lambda'_s \d v \,\d s], \quad \text{since} \>\phi_s = \frac{\lambda_s}{\lambda'_s}.
\end{align*}
The same reasoning yields $\E[\hat Q^{N^c}(H) ] = \E[\int_0^\infty\int_{\phi_s}^1 H(s,v) \lambda'_s \d v\, \d s]$. Summing the two expressions yield that $\E[\hat Q'(H)]  =  \E[\int_0^\infty\int_0^{1} H(s,u) \lambda'_s \d s\, \d u],$  and thus $\hat Q'$  has the $(\G_t \vee \F_t^{U,V})$-intensity   $\lambda'_s \d s \otimes \d u$ (\cite{bremaud1981point}).\\
Finally, Let  $ n \geq 1$, $h,t \geq 0$ and $\xi$ a $\G_{T_n^{'-}}$ measurable random variable. There exists a predictable process $Y$ such that $Y_{T'_n} = \xi$.  Then, by observing that 
\begin{equation*}
\indi_{\{{U}_n' \leq h\}}\indi_{\{T'_n \leq t\}} \xi  = \int_{\mathbb{R^+} \times [0,h]} Y_s \indi_{]T'_{n-1}, T'_n \wedge t ]}(s) \hat Q'(\d s, \d u), 
\end{equation*}
we obtain that 
\begin{align*}
& \E [\indi_{\{{U}_n' \leq h\}}\indi_{\{T'_n \leq t\}} \xi] = \E[\int_{T'_{n-1}\wedge t}^{ T'_n \wedge t} \int_0^h Y_s\lambda'_s \d s\d u] \\
&  = h  \E[\int_{T'_{n-1}\wedge t}^{ T'_n \wedge t}  Y_sdN'_s] =  h  \E[ Y_{T'_{n}}\indi_{\{T'_n \leq t\}}].
\end{align*}
This achieves to prove that ${U}_n'$ is a uniform variable independent of $\G_{T_n^-}$. By noting that $\P(U_1' \leq h_1, ... , U_n' \leq h_n) = \E[\prod_{i=1}^n \hat Q'(]T'_{i-1},T'_i]\times [0,h_i])]$, we obtain similarly that $U_1', U_2' , ...$ are i.i.d.
\end{proof}

\begin{proof}[Proof of Corollary \ref{UniquenessDomination} ] As before, we prove the one dimensional case for simplicity of notations.\\
For each $i=1,2$,  $ N^i$ is solution of  Equation \eqref{eqreciproquecomparaison}  By Proposition  \ref{PropReciproqueComparaison}, with $\lambda_t = \alpha(t, N_{t^-}^i)$.  The driving measure can be written as $Q^i= \sum_{n\geq 0} \delta_{T'_n}\delta_{U_n^{i'}}$, where $(U_n^{i'})$ is a sequence of i.i.d uniform random variables. \\
By Proposition \ref{PropReciproqueComparaison}, the distribution of $\big(U^{i'}_n, \dfrac{\alpha(T_n,\cdot)}{\lambda'_{T_n}},T'_n\big)_n$ does not depend on $i$, and thus $N^1$ and $N^2$ have the same distribution. 
\end{proof}

\section{Additional results}

\begin{lemma}
\label{LemmaTechContinuity}
Let $[x] \in \cal A^{2p}$ and let $(t^x_p)_{p\geq 1}$ its increasing sequence of jump times, with $t^x_p \rightarrow +\infty$. Let $A^x =\{[\alpha] \in \cal A^{2p} ; \; [\alpha] \prec [x]\}$.  
Then, $A^x $ is closed, and $\forall \; p\geq 1$, $\pi_{t_p^x}:[\alpha]  \in A^x \rightarrow \alpha(t^x_p)\in \N^{2p}$ is a continuous function on $A^x$. 
\end{lemma}
\begin{proof}
Let $[\alpha] \in A^x$ and $p\geq 1$. 
we recall that $[\alpha] \rightarrow \alpha(t)$ is a continuous function on $\cal A^{2p}$ (embedded with the Skorohod topology) iff $t$ is not a jump time of $\alpha$. Hence, if $t^x_p$ is not a jump time of $[\alpha]$, $\pi_{t^x_p}$ is continuous in $[\alpha]$.

When $\Delta\alpha(t^x_p) \neq 0$, then $\Delta\alpha(t^x_p) = \Delta x(t^x_p)$ since $[\alpha]\prec [x]$. By observing that $t^x_p$ is the first jump time of $\alpha(\cdot) - \alpha(\cdot \wedge t^x_p)$, we can only consider the case $p=1$, with $t^1_x$ is the first jump time of $[\alpha]$, so that   $\alpha(t_1^x) = x(t_1^x)$.

 Let $([\alpha_n])$ be a sequence in $A^x$ converging to $\alpha$, with $T^n_1$  the first jump time of $\alpha_n$. Since $[\alpha_n] \rightarrow [\alpha]$, $T^n_1 \rightarrow_n t^x_1$. For all $n\geq 0$, $[\alpha_n] \prec [x]$ , and hence the sequence $(T^1_n)$ is a subset of the discrete space $\{t^x_p ; \; p \geq 0\}$. Thus, $(T^1_n)$ is necessarily constant equal to $t^x_1$ above a given rank $N$, and for $n\geq N$, $t^x_1$ is also the first jump time of $[\alpha_n]$. This yields  that $\forall n\geq N$,  $\pi_{t^x_1}(\alpha_n) = \alpha_n (t^x_1) = x(t_1^x)= \alpha(t_1^x)$.  Therefore $\pi_{t^x_1}$ is continuous in $[\alpha]$.
\end{proof}

\begin{lemma}
\label{LemmaNtilde}
$(\tilde{\emph \bolN}^{\dem,\epsilon})$ converges stably to $\tilde \R^{\emph \dem}$ iff $(\Ndemeps)$ converges stably to a rule $\R^\dem$ such that  
for all $\mathcal O\otimes\F^{\cal A}_\infty$-measurable function $\tilde h(:=(h_s)_s)$, 
\begin{equation}
\label{EqlinkRuleNdemOmegatide}
\tilde{\R}^{\emph \dem} (\tilde h) = \R^{\emph \dem}[\int_0^\infty  h_s([\bar {\emph \bolN}^\dem]_s)\lambda^e(\d s)].
\end{equation}
\end{lemma}
\begin{proof}
Let us first assume that $(\Ndemeps)$ converges stably to a rule $\R^\dem$ in $\cR(\P, \cal A^{2p})$. Let $\tilde h\in \cal C_{bmc}(\tilde{\Omega}\times \cal A^{2p})$, and $H([\alpha]) = \int_0^\infty \tilde h(s,[\alpha]_s)\lambda^e(\d s)$, so that $\tilde{\E}[\tilde h( [\tilde \bolN^{\dem, \epsilon}])]= \E[H( [\Ndemeps])]$.\\
By letting $\epsilon \to 0$, we obtain that $(\tilde \bolN^{\dem, \epsilon})$ converges stably  in $\cR(\P\otimes \lambda^e,\cal A^{2p})$ to the rule $\tilde \R^\dem$ defined by  \eqref{EqlinkRuleNdemOmegatide}.

Reciprocally, assume that $(\tilde \bolN^{\dem,\epsilon} )$ converges stably to a rule $\tilde \R^\dem \in \cR(\P\otimes \lambda^e, \cal A^{2p})$. For $t\geq 0$, let $H_t$ be a $\G_t \otimes \F^{\cal A}_t$-measurable function in $\cal C_{bmc}(\Omega\times \cal A^{2p})$, and  $\tilde h_t(s,[\alpha]) = e^{t}\ind_{[t,\infty[}(s)H([\alpha]_t)$. Then, $\forall \epsilon >0$
\begin{equation*}
\tilde{\E}[\tilde{h}_t( [\tilde{\bolN}^{\dem,\epsilon}])] = \E[e^t\int_t^\infty H([{\bolN}^{\dem,\epsilon}]_{t\wedge s})\lambda^e(\d s)] = \E[H([\Ndemeps]_t)].
\end{equation*}
 $h_t(\tilde{\omega},\cdot) $ is  continuous on $ B_t = \{ [\alpha] ;  \;  \Delta \alpha(t) =0 \}$, and thus, we can apply the stable convergence and obtain  that  $\E[H_t(\cdot,[\Ndemeps])]$ converges for any $\G_t\otimes \F^{\cal A}_t$-measurable function $H_t \in \cal C_{bmc}(\Omega\times \cal A^{2p})$. Thus, by Proposition 3.12 in \cite{Hausler2015stable} $(\Ndemeps)$ converges stably in $\cR(\P,\cal A^{2p})$.
\end{proof}

\bibliographystyle{apalike}
\bibliography{BiblioPop2Timesalpha}

\end{document}